\documentclass{amsart} 

\usepackage[margin=1in]{geometry}
\usepackage[utf8]{inputenc}

\usepackage{amssymb,amsfonts, amsthm, enumerate}
\usepackage{mathrsfs}
\usepackage{stmaryrd}
\usepackage[mathcal]{eucal}
\usepackage{hyperref}
\usepackage{mathtools}
\usepackage{enumitem}

\swapnumbers
\theoremstyle{plain}
\newtheorem{theorem}{Theorem}[section]

\newtheorem{proposition}[theorem]{Proposition}
\newtheorem{lemma}[theorem]{Lemma}
\newtheorem{corollary}[theorem]{Corollary}

\newtheorem{theorem*}{Theorem}
\newtheorem*{corollary*}{Corollary}

\theoremstyle{definition}

\newtheorem{definition}[theorem]{Definition}
\newtheorem{examples}[theorem]{Examples}
\newtheorem{convention}[theorem]{Convention}
\newtheorem{notation}[theorem]{Notation}

\newtheorem{remark}[theorem]{Remark}
\newtheorem{remarks}[theorem]{Remarks}

\theoremstyle{remark}

\makeatletter
\renewcommand\subsection{\@startsection{subsection}{2}
  \z@{-.5\linespacing\@plus-.7\linespacing}{.5\linespacing}
  {\normalfont\scshape}}
\renewcommand\subsubsection{\@startsection{subsubsection}{3}
  \z@{.5\linespacing\@plus.7\linespacing}{-.5em}
  {\normalfont\scshape}}
\makeatother

\DeclareMathOperator{\wneg}{\mathrel{\ooalign{\hss$\neg$\hss\cr\kern0.2ex\raise1.0ex\hbox{\scalebox{0.7}{w}}}}}

\DeclareMathOperator{\card}{card}
\DeclareMathOperator{\Discrete}{Discrete}
\DeclareMathOperator{\dd}{d}
\DeclareMathOperator{\Mod}{Mod}
\DeclareMathOperator{\Sent}{Sent}
\DeclareMathOperator{\thry}{Th}
\DeclareMathOperator{\thryL}{Th_{\mathscr{L}}}

\newcommand{\cl}{\mathscr{L}_\text{cont}}
\newcommand{\bcl}{\mathscr{L}_\text{basic}}
\newcommand{\fcl}{\mathscr{L}_\text{full}}

\newcommand{\ab}{\ensuremath{a_{\bullet}}}
\newcommand{\bb}{\ensuremath{b_{\bullet}}}
\newcommand{\cb}{\ensuremath{c_{\bullet}}}
\newcommand{\sC}{\ensuremath{\mathscr{C}}}
\newcommand{\DD}{\mathbb{D}}
\newcommand{\EE}{\mathbb{E}}

\newcommand{\fb}{\ensuremath{f_{\bullet}}}

\newcommand{\Eb}{\ensuremath{E_{\bullet}}}
\newcommand{\sL}{\ensuremath{\mathscr{L}}}
\newcommand{\cM}{\ensuremath{\mathcal{M}}}
\newcommand{\cK}{\ensuremath{\mathcal{K}}}
\newcommand{\cN}{\ensuremath{\mathcal{N}}}

\newcommand{\Pb}{\ensuremath{\varphi_{\bullet}}}
\newcommand{\PCd}{PC$_{\Delta}$}
\newcommand{\RPCd}{RPC$_{\Delta}$}
\newcommand{\RR}{\ensuremath{\mathbb{R}}}
\newcommand{\tb}{\ensuremath{t_{\bullet}}}

\begin{document}

\title[Metastable convergence and logical compactness]
{Metastable convergence and logical compactness}

\author[X. Caicedo]{Xavier Caicedo}

\address{Department of Mathematics\\
  Universidad de los Andes\\
  Apartado Aéreo 4976\\
  Bogotá, Colombia}
  
\email{xcaicedo@uniandes.edu.co}

\author[E. Dueñez]{Eduardo Dueñez}

\address{Department of Mathematics\\
  The University of Texas at San Antonio\\
  One UTSA Circle\\
  San Antonio, TX 78249-0664\\
  USA}

\email{eduardo.duenez@utsa.edu}

\author[J. Iovino]{José N. Iovino}

\address{Department of Mathematics\\
  The University of Texas at San Antonio\\
  One UTSA Circle\\
  San Antonio, TX 78249-0664\\
  USA}
  
\email{jose.iovino@utsa.edu}

\date{\today}
\thanks{Dueñez and Iovino were partially funded by NSF grant DMS-1500615.
  Caicedo was partially supported by a Universidad de los Andes Science Faculty grant.}

\begin{abstract}
  The concept of metastable convergence was identified by Tao;
  it allows converting theorems about convergence into stronger theorems about uniform convergence.
The Uniform Metastability Principle (UMP) states that if $T$ is a  theorem about convergence, then the fact that $T$ is valid implies automatically that its (stronger) uniform version is valid, provided that $T$ can be stated in certain logical frameworks. 
In this paper we identify precisely the logical frameworks $\sL$ for which UMP holds.
 More precisely, we prove that the UMP holds for $\sL$ if and only if $\sL$ is a compact logic.
 We also prove a topological version of this equivalence.
We conclude by proving new characterizations of logical compactness that yield additional information about the UMP.
\end{abstract}

\maketitle

\section*{Introduction}

The concept of metastable convergence was isolated by Tao.
It played a crucial role in the proof of his remarkable result on the convergence of ergodic averages for polynomial abelian group actions~\cite{Tao:2008}, and  again in Walsh's generalization of Tao's theorem to polynomial nilpotent group actions~\cite{Walsh:2012}.

Metastability is a reformulation of the Cauchy property for sequences, i.e., a sequence in a metric space is metastable if and only if it is Cauchy.
However, for a collection of sequences, being uniformly metastable is weaker than being uniformly Cauchy.
In his 2008 paper~\cite{Tao:2008}, Tao proved a metastable version of the classical dominated convergence theorem that he then used to obtain uniform metastability rates of convergence for ergodic averages.

Tao remarked in his paper that metastability is connected to ideas from mathematical logic.
He noted, thanking U.~Kohlenbach for the observation,  that metastability is an instance of Kreisel's no-counterexample interpretation~\cite{Kreisel:1951,Kreisel:1952}, which is in turn a particular case of Gödel's Dialectica interpretation~\cite{Godel:1958}.  
In fact, before Tao's paper, the concept had been used  under different nomenclature by Avigad, Gerhardy, Kohlenbach, and Towsner in the context of proof mining. See~\cite{Avigad-Gerhardy-Towsner:2010, Kohlenbach-Lambov:2004, Kohlenbach:2005a,Kohlenbach:2008}.
For a more up-to-date survey on metastability rates obtained by proof mining, see Kohlenbach's lecture at the 2018 International Congress of Mathematicians~\cite[pp.~68--69]{Kohlenbach:2018}.

A connection between uniform metastable convergence and model-theoretic compactness was first exposed by Avigad and Iovino by using ultraproducts~\cite{Avigad-Iovino:2013}. 
After this, Dueñez and Iovino proved a metatheorem called the Uniform Metastability Principle~(\cite{Duenez-Iovino:2017}, Proposition~2.4), which roughly states the following: 

\begin{quote}
If a classical statement about convergence in metric structures is refined to a statement about  metastable convergence with some uniform rate, and this latter refinement can be expressed in the language of continuous first-order logic, then the validity of the original statement implies the validity of its uniformly metastable version.
\end{quote}

The operative word above is \emph{uniformly}: 
The striking fact about the Uniform Metastability Principle is that it allows one to convert a theorem about simple convergence into a stronger theorem about uniformly metastable convergence automatically, provided that in the statement of the theorem, one replaces convergence by the mathematically equivalent notion of metastability.
Thus, for instance, Tao's uniformly metastable dominated convergence theorem follows from the classical dominated convergence theorem as a particular application of this metatheorem.
Also, as Tao pointed out, his proposed abstract version of Walsh's ergodic theorem~\cite{Tao:Metastability} follows from the original version~\cite{Walsh:2012}.
(Tao cited the aforementioned Avigad-Iovino paper~\cite{Tao:Metastability,Avigad-Iovino:2013}.)

It is natural to ask if the Uniform Metastability Principle holds with logics more expressive than continuous first-order.
The more expressive the logic, the more powerful the metatheorem.
On the other hand, the proof of the Uniform Metastability Principle uses the fact that continuous first-order logic is compact,
and there is a delicate balance between compactness of a logic and its expressive power.

In this paper we show that the Uniform Metastability Principle is in fact equivalent to compactness for logics.
More precisely, we prove the following theorem: 

\begin{theorem*}
\label{T:metastability characterization intro} 
Let $\sL$ be a logic for metric structures. Then $\sL$ is compact if and only if every theory of convergence  is a theory of uniformly metastable convergence.
\end{theorem*}

Our main results, Theorems~\ref{thm:cpct-metastab} and~\ref{thm:cpct-metastab-count}, establish a fine
correspondence between the  many forms of compactness
arising in logic (for theories or families of theories --- see~\ref{S:compactness} of the Preliminaries section) and natural forms of the Uniform Metastability Principle (for sequences, for nets, etc.) 
Among the many forms of compactness that are studied in logic, the strongest is compactness for arbitrary theories, while the weakest is countable compactness for theories (i.e., any
countable finitely satisfiable theory is satisfiable).
At the full compactness end, the correspondence with uniform metastable convergence takes the form quoted above; at the countable compactness end, it takes the following form:

\begin{theorem*}
\label{T:metastability characterization countable intro} 
Let $\sL$ be a logic for metric structures. Then $\sL$ is countably compact if and only if every  countable theory of convergence for sequences expressible in $\sL$  is a theory of uniformly metastable convergence.
\end{theorem*}

These results provide mathematicians with a ``black box'' to convert theorems about convergence into theorems about uniform (metastable) convergence:
If a convergence theorem can be written in a logic that is compact, then its uniform metastable version is automatically true; if not, then the automatic conversion is impossible. 
A given theorem in analysis may not be expressible in first-order logic, which is a fully compact logic, but the natural framework for the theorem may be a stronger logic that admits a weaker degree of compactness, say, countable compactness.

We deal with metric structures, but even in the discrete case, i.e., for two-valued logics in discrete structures, the information given by these results appears to be new.

There are many examples of  countably compact logics (typically, extensions of first-order by generalized quantifiers --- see Examples~\ref{E:compact logics}).
Therefore, the forward implications of Theorem~\ref{T:metastability characterization intro} give a vast generalization of the Uniform Metastability Principle, and they extend the scope of this metatheorem to contexts where proof-theoretic methods may be unavailable.

We obtain the forward implications from purely topological considerations (see Section~\ref{sec:metastability}-\ref{SS:UMP topological}, where we give a topological version of the Uniform Metastability Principle).
To prove the reverse implications, we extend to the setting of metric structures a characterization of $[\kappa,\kappa]$-compactness originally proved by Makowsky and Shelah~\cite{Makowsky-Shelah:1979}, and we adapt it to characterize $(\kappa,\kappa)$-compactness.

In last section of the paper, we prove new characterizations of countable compactness for families of theories, and from this we obtain additional information about the Uniform Metastability Principle.
 To state the main result of this section, we need to recall some terminology from model theory:
If $\sL$ is a logic, a structure is $M$ is said to be \RPCd\ in $\sL$ if $M$ can be characterized, up to isomorphism, by a theory in $\sL$, possibly with the aid of additional functions and relations.
(See Definitions~\ref{D:PC} and~\ref{D:RPC}.) 

We prove the following result: 

\begin{theorem*}[Proposition~\ref{C:noncpct-rpc-omega} and Theorem~\ref{T:rpcdelta-characterizability of structures}]
\label{T:characterization compactness intro}
If ${\sL}$ is a logic for metric structures, then the following conditions are equivalent:
\begin{enumerate}
\item ${\sL}$ is not countably compact for families of theories.
\item The structure $(\omega,<)$ is \RPCd\ in ${\sL}$.
\item Any metric structure of cardinality less than the first measurable cardinal is \RPCd\ in $\sL$.
\end{enumerate}
\end{theorem*}

The equivalence between (1) and~(2) generalizes facts known for two-valued logics. 
However, the equivalence between (1) and~(3) gives us new insight on the concept of logical compactness.
This equivalence shows that logics that are not countably compact have great expressive power:
In universes where measurable cardinals do not exist, a non countably compact logic can characterize any metric structure. 

This, together with the main result, gives us the following new dichotomy for logics: 

\begin{corollary*}
If $\sL$ is a logic for metric structures, then one and only one of the following condition holds:
\begin{enumerate}
\item
The Uniform Metastability Principle holds for sequences in $\sL$.
\item
The structure $(\omega,<)$ (equivalently, any structure of power less than the first measurable cardinal) can be characterized by a theory in $\sL$, with the aid of additional functions and relations.
\end{enumerate}
\end{corollary*} 

We use this dichotomy to prove that if $\sL$ is a logic, $\kappa$  is an infinite cardinal less than the first measurable cardinal, and the Uniform Metastability Principle holds  in $\sL$ for $\kappa$-sequences, then it holds in $\sL$ for sequences (see Theorem~\ref{T:descent from nets to sequences} and Remark~\ref{nets-large-to-small}).

We do not presuppose expertise in mathematical logic from the reader.
However, we assume familiarity with the concept of \emph{structure}, as is defined in any model theory textbook, and the concepts of \emph{language} or \emph{vocabulary} of a structure.

The authors are grateful to Clovis Hamel, Ulrich Kohlenbach, and Frank Tall for invaluable comments on earlier versions of the manuscript.

\section{Preliminaries}
\label{S:prelim}

\subsection{$[0,1]$-valued structures and $[0,1]$-valued logics} 
\label{subsection:RealValuedLogics}

The formal definition of model-theoretic logic was given by P.~Lindström in his celebrated 1969 paper~\cite{Lindstrom:1969}.  
We start by recalling Lindström's classical definition.

\begin{definition} \label{Definition:LogicalSystem}
A Lindström \emph{logic} ${\sL}$ is a triple $(\mathscr{C}, \Sent_{{\sL}}, \models_{{\sL}})$, where $\mathscr{C}$ is a class of first-order structures that is closed under isomorphisms, renamings and reducts, $\Sent_{{\sL}}$ is a function that assigns to every first-order vocabulary $L$ a set $\Sent_{{\sL}}(L)$ called the set of \emph{$L$-sentences of ${\sL}$}\index{logic!sentence|ii}, and $\models_{{\sL}}$ is a binary relation between structures and sentences, such that the following conditions hold:
\begin{enumerate}
\item If $L \subseteq L'$, then $\Sent_{{\sL}}(L) \subseteq \Sent_{{\sL}}(L')$.

\item If $\mathcal {\cM} \models_{{\sL}} \varphi$ (i.e., if $\mathcal {\cM}$ and $\varphi$ are related under $\models_{{\sL}}$), then there is a vocabulary $L$ such that $\mathcal {\cM}$ is an $L$-structure  in $\mathscr{C}$ and $\varphi$ an $L$-sentence.
If ${\cM} \models_{{\sL}} \varphi$, we say that ${\cM}$ \emph{satisfies} $\varphi$, or that ${\cM}$ is a \emph{model} of $\varphi$.

\item \emph{Isomorphism Property}\index{logic!Isomorphism Property|ii}.  
If ${\cM},\mathcal{N}$ are isomorphic structures  in $\mathscr{C}$, then ${\cM} \models_{{\sL}} \varphi$ if and only if $\mathcal{N} \models_{{\sL}} \varphi$.

\item \emph{Reduct Property}\index{logic!Reduct Property|ii}.  
If $L \subseteq L'$, $\varphi$ is an $L$-sentence, and ${\cM}$ an $L'$-structure  in $\mathscr{C}$, then
\[
	{\cM} \models_{{\sL}} \varphi \qquad \text{if and only if} \qquad ({\cM} \upharpoonright L) \models_{{\sL}} \varphi.
\]

\item \emph{Renaming Property}\index{logic!Renaming Property|ii}.  
If $\rho: L \to L'$ is a renaming (i.e., a bijection $r\colon L\to L'$ that respects symbol type and and arity),  then for each $L$-sentence $\varphi$ there is an $L'$-sentence $\varphi^{\rho}$ such that ${\cM} \models_{{\sL}} \varphi$ if and only if ${\cM}^{\rho} \models_{{\sL}} \varphi^{\rho}$ for every $L$-structure ${\cM}$  in $\mathscr{C}$.  
(Here, ${\cM}^{\rho}$ denotes the structure that results from converting ${\cM}$ into an $L'$-structure through $\rho$.)
\end{enumerate}
\end{definition}

A classical first-order structure $\cM$ consists of a nonempty universe $M$ together with finitary functions and relations (or ``predicates") on $M$.
If $n$ is a nonnegative integer, any $n$-ary relation on $M$ can be seen as a function of $M^n$ into $\{0,1\}$.
In this paper we will deal with the more general concept of \emph{$[0,1]$-valued structure}, which is defined as follows:
A $[0,1]$-valued structure $\cM$ consists of a nonempty set $M$ called the \emph{universe} of $\cM$, together with finitary functions and predicates on $M$; but in this case, the predicates are $[0,1]$-valued, rather than $\{0,1\}$-valued.
A simple example of $[0,1]$-valued structure is a pseudometric space $(M,d)$ of diameter bounded by 1.
The universe of the structure is $M$ and the only predicate of the structure is $d$.

The following extension of Definition~\ref{Definition:LogicalSystem} was introduced by Caicedo and Iovino~\cite{Caicedo-Iovino:2014}:

\begin{definition}
\label{Definition:RealValuedLogic}
A \emph{$[0,1]$-valued logic} is a triple ${\sL} = (\mathscr{C}, \Sent_{{\sL}},\mathcal{V})$, 
where $\mathscr{C}$ is a class of $[0,1]$-valued structures that is closed under under isomorphisms, renamings and reducts,  $\Sent_{{\sL}}$ is a function that assigns to every first-order vocabulary $L$ a set $\Sent_{{\sL}}(L)$ called the set  of \emph{$L$-sentences of ${\sL}$}\index{logic!sentence|ii}, and $\mathcal{V}$ is a real-valued partial function on $\mathscr{C}\times\Sent_{{\sL}}$ such that the following conditions hold: 
\begin{enumerate}

\item 
If $L \subseteq L'$, then $\Sent_{{\sL}}(L) \subseteq \Sent_{{\sL}}(L')$.
\item
For every~$L$, the function $\mathcal{V}$ assigns to every pair $({\cM},\varphi)$, where ${\cM}$ is an $L$-structure in $\mathscr{C}$ and $\varphi$ is an $L$-sentence of ${\sL}$, a real number $\mathcal{V}({\cM},\varphi) = \varphi^{\cM}\in[0,1]$ called the \emph{truth value} of $\varphi$ in ${\cM}$.

\item \emph{Isomorphism Property for $[0,1]$-valued logics}.  
If ${\cM}, \mathcal{N}$ are isomorphic structures  in $\mathscr{C}$ and $\varphi$ is an $L$-sentence of ${\sL}$, then $\varphi^{\cM}=\varphi^\mathcal{N}$.

\item \emph{Reduct Property  for $[0,1]$-valued logics.}  
If $L \subseteq L'$,  $\varphi$ is an $L$-sentence of ${\sL}$, and ${\cM}$ an $L'$-structure in $\mathscr{C}$, then $\varphi^{\cM} =\varphi^{{\cM} \upharpoonright L}$.

\item \emph{Renaming Property  for $[0,1]$-valued logics.} If $\rho: L \to L'$ is a renaming, then for each $L$-sentence $\varphi$ of ${\sL}$ there is an $L'$-sentence $\varphi^{\rho}$ such that $\varphi^{\cM}=(\varphi^\rho)^{{\cM}^{\rho}}$ for every $L$-structure ${\cM}$  in $\mathscr{C}$. 
\end{enumerate}
If ${\sL}$ is a $[0,1]$-valued logic, $L$ is a vocabulary $\varphi$ is an $L$-sentence of ${\sL}$ and  ${\cM}$ is an $L$-structure in $\sC$ such that $\varphi^{\cM}=1$, we say that ${\cM}$ \emph{satisfies} $\varphi$, or that ${\cM}$ is a \emph{model} of $\varphi$, and write ${\cM} \models_{{\sL}} \varphi$.

If ${\sL}$ is a $[0,1]$-valued logic such that $\varphi^{\cM}\in\{0,1\}$ for every sentence $\varphi$ of ${\sL}$ and every structure ${\cM}$, we say that ${\sL}$ is a \emph{two-valued} logic, or a \emph{discrete} logic.
\end{definition}

\begin{definition} \label{Definitions:TheoryModelsatisfiable}
Let ${\sL}$ be a a $[0,1]$-valued logic and let $L$ be a vocabulary.
\begin{enumerate}
\item An \emph{$L$-theory} (or simply a \emph{theory} if the vocabulary is given by the context) of ${\sL}$ is a set of $L$-sentences of ${\sL}$.

\item Let $T$ be an $L$-theory of ${\sL}$.  
If ${\cM}$ is an $L$-structure such that ${\cM} \models_{{\sL}} \varphi$ for each $\varphi\in T$, we say that ${\cM}$ is a \emph{model} of $T$ and write ${\cM} \models_{{\sL}} T$.

\item A theory is \emph{satisfiable} if it has a model. 

\item If ${\cM}$ is structure of ${\sL}$, the \emph{complete ${\sL}$-theory of ${\cM}$}, denoted $\thryL({\cM})$, is the set $\{\,\varphi :{\cM} \models_{{\sL}} \varphi\,\}$.

\end{enumerate}
\end{definition}

If $L$ is a vocabulary, $\bar{x} = x_{1},\dots,x_{n}$ is a finite list of constant symbols not in $L$, and $\varphi$ is an $(L \cup \{\bar{x}\})$-sentence, we emphasize this by writing $\varphi$ as $\varphi(\bar{x})$.  
In this case we may say that $\varphi(\bar{x})$ is an \emph{$L$-formula}.  
If ${\cM}$ is an $L$-structure and $\bar{a} = a_{1},\dots,a_{n}$ is a list of elements of ${\cM}$, we write
\[
	({\cM},a_{1},\dots,a_{n}) \models_{{\sL}} \varphi(x_{1},\dots,x_{n}),
\]
or ${\cM} \models_{{\sL}} \varphi[\bar{a}]$, if the $L \cup \{\bar{x}\}$ expansion of ${\cM}$ that results from interpreting $x_{i}$ as $ a_{i}$ (for $i = 1,\dots,n$) satisfies $\varphi(\bar{x})$.

\begin{definition} \label{Definition:LogicalEquivalence}
Let ${\cM},\mathcal{N}$ be $L$-structures.  
We say that ${\cM}$ and $\mathcal{N}$ are \emph{equivalent in ${\sL}$}\index{structure!equivalent|ii}, and write ${\cM} \equiv_{{\sL}} \mathcal{N}$, if for every $L$-sentence $\varphi$ we have $\phi^{\cM}=\phi^{\cN}$.
\end{definition}

If ${\cM}$ is an $L$-structure and $A$ is a subset of the universe of~${\cM}$, we denote by $L[A]$ the expansion of the vocabulary $L$ obtained by adding distinct new constant symbols~$c_a$, one for each $a\in A$. 
We also denote by $({\cM},a)_{a\in A}$ the expansion of ${\cM}$ to an $L[A]$-structure obtained by interpreting each $c_a$ as~$a$.  
The structure  $({\cM},a)_{a\in A}$ is said to be an \emph{expansion of ${\cM}$ by constants}.

\begin{definition} \label{Definition:ElementarySubstructure}
Let ${\sL}$ be a $[0,1]$-valued logic and let ${\cM},\mathcal{N}$ be $L$-structures with ${\cM}$ a substructure of $\mathcal{N}$.  
We say that ${\cM}$ is an \emph{${\sL}$-substructure}\index{structure!elementary substructure|ii} of $\mathcal{N}$ or that $\mathcal{N}$ is an \emph{${\sL}$-extension} of ${\cM}$, and we write ${\cM} \preceq_{{\sL}} \mathcal{N}$, if  $({\cM},a)_{a\in {\cM}}\equiv_{{\sL}} (\mathcal{N},a)_{a\in {\cM}}$.
\end{definition}

\begin{convention}
In order to avoid clutter in the notation, if a logic ${\sL}$ is fixed, we may suppress the subindex in symbols like $\equiv_{\mathscr L}$, $\prec_{\mathscr L}$,  $\thryL(\cdot)$, etc.  
\end{convention}

\subsection{Connectives and classical quantifiers} 
\label{subsection:Connectives}

\begin{definition} \label{Definition:LukasiewiczImplication}
The \emph{Łukasiewicz implication} is the function $\to$ from $[0,1]^2$ into $[0,1]$ defined by
\[
x\to y= \min\{1-x+y,1\}\quad\text{for all $(x,y)\in[0,1]^2$.}
\]
\end{definition}

Note that $x\to y$ has the value~$1$ if and only if $x\le y$.

\begin{definition}
\label{D:basic connectives}
We will say that a $[0,1]$-valued logic ${\sL}$ is \emph{closed under the basic connectives} if the following conditions hold for every vocabulary $L$:
\begin{enumerate}
\item If $\varphi,\psi \in \Sent_{{\sL}}(L)$, then there exists a sentence $\varphi\to\psi$ in $\Sent_{{\sL}}(L)$ such that $(\varphi\to\psi)^{\cM} = \varphi^{\cM}\to\psi^{\cM}$ for every $L$-structure ${\cM}$.
\item For each rational $r\in[0,1]$, the set $\Sent_{{\sL}}(L)$ contains a sentence with constant truth value~$r$.  
  These sentences are called the \emph{Pavelka constants} of ${\sL}$.%
\footnote{The Pavelka constants are not needed (see~\cite{Caicedo:2017}), but they simplify the exposition.}
\end{enumerate}
\end{definition}

\begin{notation}
If ${\sL}$ is a $[0,1]$-valued logic that is closed under the basic connectives, $\varphi$ is a sentence of ${\sL}$, and $r$ is a Pavelka constant of ${\sL}$, we will write $\varphi\le r$, $\varphi\ge r$ and $\varphi = r$ as abbreviations, respectively, of $\varphi\to r$, $r\to \varphi$ and $(\varphi\to r) \wedge (r\to\varphi)$. 
\end{notation}

\begin{remark}
\label{R:reduction of [0,1]-valued to 2-valued}
Let ${\sL}$ be a $[0,1]$-valued logic and let $L$ be a vocabulary.  
If ${\cM}$ is an $L$-structure of ${\sL}$, $\varphi$ is an $L$-sentence of ${\sL}$, and $r$ is a Pavelka constant of ${\sL}$, then  ${\cM}\models_{{\sL}}\varphi\le r$ if and only if $\varphi^{\cM}\le r$, and ${\cM}\models_{{\sL}}\varphi\ge r$ if and only if $\varphi^{\cM}\ge r$;
thus, the truth value $\varphi^{\cM}$ is determined by either of the sets
\[
\{\,r\in\mathbb{Q}\cap[0,1] : {\cM}\models_{{\sL}}\varphi\le r\,\},\qquad
\{\,r\in\mathbb{Q}\cap[0,1] : {\cM}\models_{{\sL}}\varphi\ge r\,\}.
\]
\end{remark}

\begin{notation}
\label{N:connectives}
If ${\sL}$ is a $[0,1]$-valued logic that is closed under the basic connectives and $\varphi,\psi$ are sentences of ${\sL}$, we write $\neg\varphi$ and $\varphi\lor\psi$ as abbreviations, respectively, of $\varphi\to 0$ and $(\varphi\to\psi)\to\psi$, and $\varphi\land\psi$ as an abbreviation of $\neg(\neg\varphi\lor\neg\psi)$.
\end{notation}

Note that  for every $L$-structure ${\cM}$, one has
\begin{align*}
(\varphi\le 0)^{\cM} &= 1 - \varphi^{\cM},\\
(\varphi\lor\psi)^{\cM} &= \max\{\varphi^{\cM},\psi^{\cM}\},\\
(\varphi\land\psi)^{\cM} &= \min\{\varphi^{\cM},\psi^{\cM}\}.
\end{align*}
In particular, every $[0,1]$-valued logic that is closed under the basic connectives is closed under conjunctions and disjunctions. 
On the other hand, $\cM\models\varphi$ implies $\cM\not\models\neg\varphi$, but not conversely.
We call $\neg\varphi$ the \emph{Łukasiewicz negation} or \emph{weak negation} of~$\varphi$.

We will refer to any function from $[0,1]^n$ into $[0,1]$, where $n$ is a nonnegative integer, as an $n$-ary \emph{connective}.  
The Łukasiewicz implication and the Pavelka constants are continuous connectives, as are all the projections $(x_1,\dots, x_n)\mapsto x_i$.  
The following proposition states that any other other continuous connective can be approximated by finite combinations of these.

\begin{proposition}\label{P:connectives approximation}
Let $\mathscr{C}$ be the class of connectives generated by composing the Łukasiewicz implication, the Pavelka constants, and the projections.  
Then every continuous connective is a uniform limit of connectives in $\mathscr{C}$. 
\end{proposition}

\begin{proof}
Since $\mathscr{C}$ is closed under the connectives $\max\{x,y\}$ and $\min\{x,y\}$, by the Stone-Weierstrass theorem for lattices \cite[pp.~241-242]{Gillman-Jerison:1976}, we only need to show that the connectives $rx$, where $r$ is a dyadic rational, can be approximated by connectives in $\mathscr{C}$. 

Notice that if $x\in[0,1]$,
\[
\frac{1}{2}x=\lim_n \bigvee\limits_{i=1}^{n}\left(\frac{i}{n}\wedge \lnot \left(x\to \frac{i}{n}\right)\right).
\]
Hence, since the truncated sum $a\oplus b = \min(a+b,1) = \neg x\to y$ is in $\mathscr{C}$, so are all the connectives $\big(\frac{1}{2}x+\dots+\frac{1}{2^n}\big)x$, for any positive integer $n$.
\end{proof}

\begin{definition}\label{D:closed under classical quantifiers}
Let ${\sL}$  be a $[0,1]$-valued logic.  
We say that ${\sL}$ is \emph{closed under existential quantifiers} if given any $L$-formula $\varphi(x)$ there exists an $L$-formula $\exists x\varphi$ such that for every $L$-structure ${\cM}$ one has $(\exists x\varphi)^{\cM}=\sup_{a\in M}(\varphi[a]^{\cM})$.  
Similarly, we say that ${\sL}$ is \emph{closed under universal quantifiers} if given any $L$-formula $\varphi(x)$ there exists an $L$-formula $\forall x\varphi$ such that for every $L$-structure ${\cM}$ one has $(\forall x\varphi)^{\cM}=\inf_{a\in M}(\varphi[a]^{\cM})$.
\end{definition}

\subsection{Metric structures and logics for metric structures}
\label{S:Logics for metric structures}

\begin{definition}
A \emph{metric structure} is a $[0,1]$-valued structure $\cM$ such that one of the predicates of $\cM$ is a metric $d$ on the universe of $\cM$, and all the functions and predicates of $\cM$ are uniformly continuous with respect to $d$.
\end{definition}

Note that classical structures are metric structures; we regard them as being endowed with the discrete metric.
The predicate for this metric is $\neg(x=y)$.
For this reason, we refer to classical structures as \emph{discrete} structures.

\begin{definition}
\label{D:logic for metric structures}
A \emph{logic for metric structures} is a $[0,1]$-valued logic $\sL$ such that the structures of $\sL$ are metric structures and $\sL$ is closed under the basic connectives and the existential and universal quantifiers (see Definitions~\ref{D:basic connectives} and~\ref{D:closed under classical quantifiers}).
\end{definition}

\begin{remark}
  To any logic~$\sL$ for metric structures there corresponds a logic $\widetilde{\sL}$ for discrete structures, i.e., for models of the sentence
  \begin{equation*}
    \forall x\forall y (\dd(x,y)=0 \vee \dd(x,y)=1).
  \end{equation*}
  It follows trivially from the definition of logic for metric structures that $\widetilde{\sL}$ extends classical (discrete) first-order logic~$\sL_{\omega\omega}$.
\end{remark}

\subsection{Examples of logics for metric structures}
\label{S:basic examples}
 
In Subsection~\ref{S:compactness} we shall turn our attention to compactness. 
Examples (2)--(5) below are examples of compact logics.

1. \textbf{Two-valued logics}. Certainly, any Lindström logic that is closed under the Boolean connectives and classical quantifiers can be seen as a two-valued logic and hence as a logic for metric structures. 

2. \textbf{Basic continuous logic}.  
This logic, which we will temporarily denote as $\bcl$, is defined in the following manner.  
The class of structures of $\bcl$ is the class of all metric structures.  
The class of sentences of  $\bcl$ is defined as follows.  
For a vocabulary $L$, the concept of $L$-term is defined as in first-order logic.  
If $t(x_1,\dots,x_n)$ is an $L$-term (where $x_1,\dots, x_n$ are the variables that occur in $t$), ${\cM}$ is an $L$-structure, and $a_1,\dots,a_n$ are elements of the universe of ${\cM}$, the interpretation $t^{\cM}[a_1,\dots,a_n]$ is defined as in first-order logic as well.  
The atomic formulas of $L$ are all the expressions of the form $\dd(t_1,t_2)$ or $R(t_1,\dots,t_n)$, where  $R$ is an $n$-ary predicate symbol of $L$.  
If $\varphi(x_1,\dots,x_n)$ is an atomic $L$-formula with variables $x_1,\dots,x_n$ and $a_1,\dots,a_n$ are elements of an $L$-structure ${\cM}$, the interpretation $\varphi^{\cM}[a_1,\dots,a_n]$ is defined naturally by letting
\[
  R(t_1,\dots,t_n)^{\cM}[a_1,\dots,a_n] = R^{\cM}(t_1^{\cM}[a_1,\dots,a_n],\dots,t_n^{\cM}[a_1,\dots,a_n])
\]
and
\[
 \dd(t_1,t_2)^{\cM}[a_1,\dots,a_n] = d^{\cM}(t_1^{\cM}[a_1,\dots,a_n],t_2^{\cM}[a_1,\dots,a_n])
 \]
(where $d^{{\cM}}$ is the metric in~${\cM}$).
The $L$-formulas of $\bcl$ are the syntactic expressions that result from closing the atomic formulas of $L$ under the Łukasiewicz implication, the Pavelka constants, and the existential quantifier.
A \emph{sentence} of $\bcl$  is a formula without free variables, and the \emph{truth value} of a $L$-sentence $\varphi$ in an $L$-structure ${\cM}$ is $\varphi^{\cM}$.
We write ${\cM}\models_{\bcl}\varphi[a_1,\dots,a_n]$ if $\varphi[a_1,\dots,a_n]^{\cM}=1$. 

Recall that in any $[0,1]$-valued logic that is closed under the basic connectives,  the expressions $\neg\varphi$, $\varphi\lor\psi$, $\varphi\land\psi$, $\varphi\le r$, and $\varphi\ge r$ are written as abbreviations of $\varphi\to 0$, $(\varphi\to\psi)\to\psi$,  $\neg(\neg\varphi\lor\neg\psi)$, $\varphi\to r$, and $r\to\varphi$, respectively.  
In $\bcl$ we also regard $\forall x\varphi$ as an abbreviation of $\neg\exists x\neg\varphi$.

3. \textbf{Full continuous logic}.  
This logic, temporarily denoted~$\fcl$, is the same as~$\bcl$ above with the difference that, instead of taking the closure under the Łukasiewicz-Pavelka connectives, one takes the closure under all continuous connectives (and the existential quantifier). 

Proposition~\ref{P:connectives approximation} yields the following remark, which allows one to transfer model-theoretic results between $\bcl$ and $\fcl$: 
\begin{remark}
\label{R:formulas approximation}
For every $L$-formula $\varphi(\bar x)$ of $\fcl$ and for every $\epsilon>0$ there exists a formula~$\psi(\bar x)$ of $\bcl$ such that $|\varphi^{\cM}[\bar a]-\psi^{\cM}[\bar a]| \le \epsilon$ for every complete $L$-structure ${\cM}$ and every tuple $\bar a$ in the universe $M$ of~$\cM$ with $\ell(\bar a)=\ell(\bar x)$.  
It follows that if ${\cM},\mathcal{N}$ are $L$-structures, then ${\cM}\equiv_{\bcl} \mathcal{N}$ if and only if ${\cM}\equiv_{\fcl} \mathcal{N}$, and ${\cM}\preceq_{\bcl} \mathcal{N}$ if and only if ${\cM}\preceq_{\fcl} \mathcal{N}$.  
Moreover, every structure is equivalent in $\bcl$ (and $\fcl$) to its metric completion.
\end{remark}

4. \textbf{The continuous logic framework of Ben Yaacov and Usvyatsov}.  
This logic is the restriction of $\fcl$ to the class of complete metric structures;
it was introduced by Ben Yaacov and Usvyatsov~\cite{Ben-Yaacov-Usvyatsov:2010} as a reformulation of Henson's logic for metric spaces, based on the concept of continuous model theory developed by Chang and Keisler~\cite{Chang:1961,Chang-Keisler:1966}.

5. \textbf{Łukasiewicz-Pavelka logic}.  
The formulas of Łukasiewicz-Pavelka logic are like those of basic continuous logic, with the following difference:
in place of the distinguished metric~$d$, one uses the \emph{similarity} relation $x\approx y$.  
However, there is a precise correspondence between the two relations, namely, $\dd(x,y)$ is $1-(x\approx y)$ (in other words, the two relations are weak negations of each other)---see Section~5.6 of~\cite{Hajek:1998}, especially Example~5.6.3-(1).  
Also, in Łukasiewicz-Pavelka logic, for each $n$-ary operation symbol $f$, one has the axiom
\label{P:similarity axioms}
\[
(x_1\approx y_1\land\dots\land x_n\approx y_n) \to 
(f(x_1\dots,x_n)\ \approx f(y_1,\dots y_n)),
\]
and similarly, for each $n$-ary predicate symbol $R$, one has the axiom
\[
(x_1\approx y_1\land\dots\land x_n\approx y_n) \to 
(R(x_1\dots,x_n) \leftrightarrow R(y_1,\dots y_n)),
\]
where `$\varphi\leftrightarrow\psi$' abbreviates `$(\varphi\rightarrow\psi)\wedge(\psi\rightarrow\varphi)$'
(\cite{Hajek:1998}, Definition~5.6.5).
Thus,  Łukasiewicz-Pavelka logic is the restriction of basic continuous logic to the class of 1-Lipschitz structures, i.e., structures whose  operations and predicates are 1-Lipschitz.

Historically, Pavelka extended Łukasiewicz propositional logic by adding the rational constants, and proved a form of approximate completeness for  the resulting logic.  
See~\cite{Pavelka:1979I,Pavelka:1979II,Pavelka:1979III} (see also Section 5.4 of~\cite{Hajek:1998}.)
This is known as Pavelka-style completeness.  
Łukasiewicz-Pavelka logic  is also referred to in the literature as rational Pavelka logic, or Pavelka many-valued logic.  
Novák proved Pavelka-style completeness for predicate Łukasiewicz-Pavelka logic, which he calls ``first-order fuzzy logic", first using ultrafilters~\cite{Novak:1989,Novak:1990}, and later using a Henkin-type construction~\cite{Novak:1995}.  
Another proof of  Pavelka-style completeness for predicate Łukasiewicz-Pavelka logic was given by Hajek (\cite {Hajek:1997} and~\cite[Section 5.4]{Hajek:1998}).  

6. \textbf{Infinitary $[0,1]$-valued logics}.
Different $[0,1]$-valued logics with infinitary formulas have been been studied  by Ben Yaacov-Iovino~\cite{Ben-Yaacov-Iovino:2009}, Eagle~\cite{Eagle:2014,Eagle:2017}, Grinstead~\cite{Grinstead:2013}, Sequeira~\cite{Sequeira:2013}, and Caicedo~\cite{Caicedo:2017}.  
See~\cite{Caicedo:2017} and~\cite{Eagle:2017} for comparisons among these.

\begin{convention}
Throughout the rest of the paper, the symbol $\cl$ will denote any of the logics in Examples (1)--(5) above.
\end{convention}

\subsection{Relativizations}

The fact that a given predicate of a $[0,1]$-structure (including the metric of a metric structure) takes on values in $\{0,1\}$, can be expressed using only the connectives $\lor$ and $\neg$:

\begin{definition} \label{Definition:DiscretePredicate}
Let ${\sL}$ be a $[0,1]$-valued logic that is closed under the basic connectives and let ${\cM}$ be an $L$-structure of ${\sL}$.  
Let $P$ a predicate symbol of $L$ or the symbol denoting the metric. 
We define $\Discrete(P)$ to be the $L$-formula
\[
	\forall \bar{x}(P(\bar{x}) \vee \neg P(\bar{x})),
\]
and call $P^{{\cM}}$ \emph{discrete} if ${\cM} \models \Discrete \big(P)$.
\end{definition}

Let $L$ be a vocabulary and let $P(x)$ be a monadic predicate not in $L$.  
 If ${\cM}$ is an $(L\cup\{P\})$-structure with universe $M$ such that $P^{{\cM}}$ is discrete, and a valid $L$-structure of $\mathcal{L}$ is obtained by restricting the universe of ${\cM}$ to $\{a\in M: {\cM}\models_{\mathcal{L}}P[a]\}$, then we denote this structure by ${\cM}\upharpoonright \{x : P(x)\}$ or ${\cM}\restriction P$.
Note that if ${\cM}$ is complete, the continuity of $P$ ensures that ${\cM}\restriction P$, when defined, is complete.  

\begin{definition}\label{D:logic permits relativization single predicate}
 A $[0,1]$-valued logic ${\sL}$ \emph{permits relativization to discrete predicates} if for every vocabulary~$L$, every $L$-sentence $\varphi$, and every monadic predicate symbol $P$ not in $L$ there exists an $(L\cup\{P\})$-sentence, denoted $\varphi^P$ or $\varphi^{\{x: P(x)\}}$ and called the \emph{relativization} of $\varphi$ to $P$, such that the following holds:
 If ${\cM}$ is an $(L\cup\{P\})$-structure with universe $M$ such that $P^{{\cM}}$ is discrete, then
 \begin{itemize}
\item
${\cM}\upharpoonright \{x : P(x)\}$ is an $(L\cup\{P\})$-structure of ${\sL}$.
\item
For all $c\in M$,
\[
(\varphi^{P})^{\cM}[c]=
\varphi^{{\cM}\upharpoonright P}[c].
\]
\end{itemize}
 \end{definition}

As an example, if $\varphi$ is an formula of $\cl$,  the relativization of $\varphi$ to $P$ can be  defined by the following recursive rule:

\begin{itemize}
\item
If $\varphi$ is atomic, then $\varphi^{P}$ is $\varphi$.
\item
If $\varphi$ is of the form $C(\psi_1,\dots,\psi_n)$, where $C$ is a connective, then $\varphi^{P}$ is $C(\psi_1^{P},\dots,\psi_n^{P})$.
\item
If $\varphi$ is of the form $\exists y \psi$, then $\varphi^{P}$ is $\exists y(P(y) \land \psi^{P})$.
\item
If $\varphi$ is of the form $\forall y \psi$, then $\varphi^{P}$ is $\forall y(\neg P(y) \lor \psi^{P})$.
\end{itemize}

One may verify that all the basic examples of $[0,1]$-valued logics discussed in Subsection~\ref{S:basic examples} satisfy the following stronger property:

\begin{definition}\label{def:relativns}
 A $[0,1]$-valued logic ${\sL}$ \emph{permits relativization to definable families of predicates} if for every vocabulary $L$, every $L$-sentence~$\varphi$, every binary predicate symbol $R$ not in $L$,  and any variable $y$, there is an $(L\cup\{R\})$-formula $\psi(y)$, denoted $\varphi^{\{x: R(x,y)\}}$ or $\varphi^{R(\cdot,y)}$,
such that the following holds: 
Whenever ${\cM}$ is an $(L\cup\{R\})$-structure with universe $M$ such that for every $b\in M$,
\begin{itemize}
\item
Either ${\cM}\models R[a,b]$ or ${\cM}\models \neg R[a,b]$ for every $a\in M$ (i.e., the collection $\{R^{{\cM}}(\cdot,b) : b\in M\}$ consists of discrete predicates), and
\item
${\cM}\upharpoonright \{x : R(x,b)\}$ (also denoted ${\cM}\restriction R(\cdot,b)$) is defined as an $L$-structure of~${\sL}$, 
\end{itemize}
one has, 
 \[
(\varphi^{R(\cdot,b)})^{\cM} = \varphi^{{\cM}\upharpoonright R(\cdot,b)}
\qquad\text{for all $b\in {\cM}$.}
\]
Such formula $\varphi^{R(\cdot,y)}$ is a \emph{relativization of $\varphi$ by~$R(x,y)$ with parameter~$y$}.
\end{definition}

\begin{definition}
\label{D:regular-logic}
We will say that a logic for metric structures is \emph{regular} if it permits relativization to definable families of predicates.
\end{definition}

All the logics mentioned in Subsection~\ref{S:basic examples} and in Examples~\ref{E:compact logics} of the next section are regular.

\subsection{$[\kappa,\lambda]$-compactness and $(\kappa,\lambda)$-compactness}
\label{S:compactness}

Recall that if $(X,d)$ and $(Y,\rho)$ are pseudometric spaces and $F:X\to Y$ is uniformly continuous, a \emph{modulus of uniform continuity for $F$} is a function $\Delta:(0,\infty)\to[0,\infty)$ such that, for all $x,y\in B$ and $\epsilon > 0$,
\[
\dd(x,y) < \Delta(\epsilon) \quad\Rightarrow\quad \rho(F(x),F(y))\le\epsilon.
\]

\begin{definition}\label{def:kappa-lambda}
If ${\sL}$ is a logic for metric structures and $T$ is an ${\sL}$-theory, the class of models of $T$ will be denoted $\Mod_{\mathscr L}(T)$.  
An \emph{${\sL}$-elementary class} is a class of the form $\Mod_{\mathscr L}(\varphi)$, where $\varphi$ is a sentence.
\end{definition}

\begin{definition}
\label{D:uniform class}
Let $L$ be a vocabulary and let  $\mathscr{C}$  be a class of $L$-structures.  

We will say that $\mathscr{C}$  is a \emph{uniform class} if for every function symbol $f$ of $L$ there exists $\Delta_{f }:(0,\infty)\to[0,\infty)$ such that for every structure $\mathcal M$ of $\mathscr{C}$, the function $\Delta$ is a modulus of uniform continuity for $f^{{\cM}}$.  
The collection $(\Delta_f)_{f\in L}$ is called a \emph{modulus of uniform continuity for $\mathscr{C}$}.

If ${\sL}$ is a logic for metric structures and $T$ is an ${\sL}$-theory, we will say that $T$ is \emph{uniform} if $\Mod_{\mathscr L}(T)$ is a uniform class.  
We will say that a family $\mathcal{T}$ of $L$-theories is \emph{uniform} if there exists a common modulus of uniform continuity for $\Mod_{\mathscr L}(T)$ for all $T\in \mathcal{T}$.
\end{definition}

\begin{remark}
The definition of uniform class given above applies only to $[0,1]$-valued structures.  
For unbounded or non-uniformly bounded structures, a more general definition imposing local bounds in addition to local moduli of continuity is needed~\cite{Duenez-Iovino:2017}.
\end{remark}

\begin{definition}
\label{D:compactness}
Let ${\sL}$ be a logic for metric structures and let $\kappa,\lambda$ be infinite cardinals with $\lambda\le\kappa\le \infty$.  
\begin{enumerate}
\item
We will say that ${\sL}$ is \emph{$[\kappa,\lambda]$-compact} if the following holds: 
Whenever $L$ is a vocabulary and $\mathcal{T}$ is a uniformly continuous family of $L$-theories of ${\sL}$ of cardinality at most~$\kappa$, the union $\bigcup \mathcal{T}$ is satisfiable if $\bigcup \mathcal{T}_0$ is satisfiable for every subfamily $\mathcal{T}_0\subseteq \mathcal{T}$ of cardinality strictly less than~$\lambda$.  
We will say that $\mathscr L$ is \emph{compact} if and only if $\sL$ is $[\infty,\omega]$-compact, i.e., $[\kappa,\omega]$-compact for every $\kappa$.
\item 
We will say that ${\sL}$ is \emph{$(\kappa,\lambda)$-compact} if the following holds: 
Whenever $L$ is a vocabulary and $T$ is a uniform $L$-theory of ${\sL}$ of cardinality at most~$\kappa$,
we have that $T$ is satisfiable if every subtheory of $T$ of cardinality strictly less than~$\lambda$ is satisfiable. 
\end{enumerate}
\end{definition}

\begin{remark}\label{rem:contrast-countable-compactnesses}
Clearly, $[\kappa,\lambda]$-compactness is stronger than $(\kappa,\lambda)$-compactness. 
However, the two properties become equivalent if, in the definition of $[\kappa,\lambda]$-compactness, we consider only theories of cardinality at most~$\kappa$.
Also, both are equivalent when $\kappa=\infty$.
\end{remark}

\begin{examples}
\label{E:compact logics}
Let $\sL_{\omega\omega}$ be first-order logic.
Given a quantifier $Q$, we denote by $\sL_{\omega\omega}(Q)$ the extension of  $\sL_{\omega\omega}$ by the quantifier $Q$.
\begin{enumerate}
\item
If $\kappa$ is an infinite cardinal and $\exists^{\ge\kappa}$ is the quantifier that says ``there exist $\kappa$-many'', then 
$\sL_{\omega\omega}(\exists^{\ge\kappa})$ is  $[\omega,\omega]$-compact for $\kappa = (2^\omega)^+$, 
and in general $[\delta,\delta]$-compact for any $\kappa$ of the form $(\lambda^{\delta})^+$~\cite{Fuhrken:1965}.
In particular, first-order logic extended with the quantifier  ``there exist at most continuum many'' is $[\omega,\omega]$-compact.
The logic $\sL_{\omega\omega}(\exists^{\ge\aleph_1})$ (i.e., first-order extended with the quantifier ``there exist uncountably many'') 
is known for its good behavior~\cite{Keisler:1970,Kaufmann:1985}.
This logic is $(\omega, \omega)$-compact~\cite{Vaught:1964} but not $[\omega, \omega]$-compact~\cite{Caicedo:1999}.
It is consistent to assume that $\sL_{\omega\omega}(\exists^{\ge\kappa})$ is $(\omega,\omega)$-compact for all $\kappa>\omega$. See~\cite{Shelah-Vaananen} for a detailed account.

\item
\emph{Stationary logic} is  the extension of first-order with the second-order quantifier that  says ``for almost all countable sets'' (more precisely, for a close unbounded family of subsets of the universe).
This logic is $(\omega,\omega)$-compact.
Stationary logic was introduced by Shelah~\cite{Shelah:1971c,Shelah:1972}, investigated in detail by Barwise-Kaufmann-Makkai~\cite{Barwise-Kaufmann-Makkai:1978} and further by Mekler-Shelah~\cite{Mekler-Shelah:1985,Mekler-Shelah:1986}).
\item
Shelah's cofinality quantifier $Q^{\rm cof}_{\omega}$ such that $Q^{\rm cof}_{\omega}x,y\,\varphi(x,y)$ says ``$\varphi(x,y)$ defines a linear order of cofinality~$\omega$'' gives a fully compact proper extension of first-order logic~\cite{Shelah:1971c}.
This logic is a sublogic of stationary logic~\cite[Lemma 4.4]{Shelah:1985}.
\item
Other compact extensions of first-order by second order quantifiers have been studied by Shelah~\cite{Shelah:1975,Shelah:Compact} and Mekler-Shelah~\cite{Mekler-Shelah:1993}.
\item
The infinitary logic $\sL_{\kappa\kappa}$ is $(\kappa,\kappa)$-compact if and only if $\kappa$ is weakly compact, and it is $(\infty,\kappa)$-compact if and only if $\kappa$ is strongly compact.
\end{enumerate}
\end{examples}

\begin{definition}
\label{D:compactness-top}
Let $X$ be a topological space.  
If  $\kappa,\lambda$ are infinite cardinals with $\lambda\le\kappa$, a topological space is said to be \emph{$[\kappa,\lambda]$-compact} if whenever $\mathcal F$ is a family of at most $\kappa$ closed sets such that the intersection of any subfamily of $\mathcal F$ of cardinality less than $\lambda$ has nonempty intersection we must have $\bigcap \mathcal{F}\neq \emptyset$.%
\footnote{This concept was introduced in topology by Alexandroff and Urysohn~\cite{Alexandroff-Urysohn:1929}.  
 Smirnov took up its systematic study and considered variations of the definition~\cite{Smirnov:1950,Smirnov:1951}, as done later by Gál~\cite{Gal:1957,Gal:1958} and Noble~\cite{Noble:1971}.
For a comprehensive introduction (including detailed comparisons among the variants introduced by Smirnov, Gaal, and Noble), see Vaughan's paper~\cite{Vaughan:1975}.}
Note that a topological space $X$ is compact if and only if $X$ is $[\infty,\omega]$-compact. 
\end{definition}

\begin{remarks}
\label{R:countable compactness of logics}
If $L$ is a vocabulary, the class of $L$-structures of ${\sL}$ can be regarded as a topological space naturally by letting the classes of the form $\Mod_{\mathscr L}(T)$, where $T$ is an $L$-theory, be the closed classes of the topology (in other words, elementary ${\sL}$-classes are the basic closed sets).  

Clearly, ${\sL}$ is $[\kappa,\lambda]$-compact if and only if the class of $L$-structures of~$\sL$ is $[\kappa,\lambda]$-compact for every vocabulary~$L$.
In particular, ${\sL}$ is $[\omega,\omega]$-compact if and only if, for every vocabulary $L$, the class of $L$-structures of ${\sL}$ is countably compact.
It is easy to verify that if $\lambda<\kappa$, then $[\kappa,\lambda]$-compactness is equivalent to $[\delta,\delta]$-compactness for all regular $\delta$ such that $\lambda\le\delta\le\kappa$ (see~\cite{Vaughan:1975}).

For general compact extensions of $\sL_{\omega\omega}$ from a topological viewpoint see~\cite{Caicedo:1999,Caicedo:1993}.

\end{remarks}

\begin{remark}
The nomenclature for square-bracket compactness is not unified in logic and topology.
 The term ``countable compactness'' corresponds to $[\omega,\omega]$-compactness in topology and to $(\omega,\omega)$-compactness in logic.
 Also, $[\lambda,\kappa]$-compactness in topology, corresponds to $[\kappa,\lambda]$-compactness in logic.
 For the rest of this paper, we will adhere to the usage within logic.
\end{remark}

\section{$[\kappa,\kappa]$-compactness and cofinality}
\label{sec:k-l-cpct-cofinal}

Recall that a logic for metric structures is \emph{regular} if it permits relativization to definable families of predicates (see Definitions~\ref{def:relativns} and~\ref{D:regular-logic}).

The following theorem is a version for metric structures of a theorem of Makowsky and Shelah~\cite{Makowsky-Shelah:1979}.

\begin{theorem}
\label{thm:k-k-cpct-cof}

Let $\sL$ be a logic for metric structures and let $\kappa$ be a regular cardinal.
Then (1) below implies (2).
If $\sL$ is regular, (2) implies (1).
\begin{enumerate}
\item ${\sL}$ is $[\kappa,\kappa]$-compact.
\item If $L$ is a vocabulary containing a monadic predicate symbol~$P$, a binary predicate symbol~$\lhd$, and a family $(\overline{\alpha}:\alpha< \kappa)$ of constant symbols, then every satisfiable uniform theory of~$\sL$ extending the theory~$T_{\kappa}$ consisting of the sentences
\begin{itemize}
\item $\Discrete(P)$, $\Discrete(\lhd)$, plus
\item expressing that $\lhd$ is a linear order on the truth set of~$P$, and
\item $P(\overline{\alpha})$, $\overline{\alpha}\lhd \overline{\beta}$, for $\alpha<\beta<\kappa$,
\end{itemize}
necessarily has a model~$\cM$ such that $(\overline{\alpha}^{\cM}:\alpha<\kappa)$ is not cofinal in~$(P^{\cM},\lhd^{\cM})$.
\end{enumerate}
\end{theorem}

\begin{proof}
$(1)\Rightarrow(2)$:\quad 
Let $\sL$ be $[\kappa,\kappa]$-compact and let $T$ be a satisfiable uniform theory extending~$T_{\kappa}$.
 For $\delta<\kappa$ and a new constant~$c$, the theory $T_{\delta} = T \cup \{P(c)\} \cup \{\overline{\alpha} < c : \alpha<\delta\}$ has a model (e.g., the expansion of a model $\cM$ of~$T_{\kappa}$ obtained upon interpreting $c$ by $\overline{\delta}$). 
By the hypothesis of $[\kappa,\kappa]$-compactness, $\bigcup_{\delta<\kappa}T_{\delta}$ is satisfiable and thus has a model that evidently satisfies the requirements.

$(2)\Rightarrow(1)$:\quad
Assume that $\kappa$ is regular and ${\sL}$ is not $[\kappa,\kappa]$-compact.
Fix a uniform family $\mathcal{T} = \{T_\alpha\}_{\alpha<\kappa}$ of $\kappa$ many $L$-theories of ${\sL}$ such that $\bigcup \mathcal{T}'$ is satisfiable for every subfamily $\mathcal{T}'\subseteq\mathcal{T}$ having strictly fewer than $\kappa$ elements, but $\bigcup \mathcal{T}$ is not satisfiable.  
Without loss of generality, we can assume $T_\alpha\subseteq T_\beta$ for $\alpha<\beta<\kappa$.  
We can also assume that every $T_\alpha$ contains sentences specifying the uniform continuity modulus for $\mathcal{T}$ (which is common to all $T_{\alpha}$).
Let $L'$ extend $L$ with new symbols $P$, $\lhd$, and $(\overline{\alpha}:\alpha<\kappa)$ per the hypotheses of~(2), plus a binary predicate symbol~$R$.
Let $T = T_{\kappa} \cup \{\Discrete(R)\} \cup \bigl\{\forall y\bigl(P(y)\wedge \overline{\alpha}\lhd y\rightarrow \varphi^{R(\cdot,y)}\bigr) : \varphi\in T_{\alpha}, \alpha<\kappa\bigr\}$.

We construct a model of~$T$ as follows.
For each $\alpha<\kappa$, let $\mathcal M_\alpha$ be a model of~$T_\alpha$.  
Consider the structure $(\kappa,<)$ as a discrete linear order.  
Let $\cM$ be the $L'$-structure such that:
\begin{itemize}
\item The universe $M$ of $\mathcal M$ is the disjoint union $\kappa \sqcup \bigsqcup_{\alpha<\kappa}M_\alpha$.
\item $\overline{\alpha}^{\mathcal M}=\alpha$ for each $\alpha<\kappa$.
\item The distance between elements of $\bigsqcup_{\alpha<\kappa}M_\alpha$ in the same $M_\alpha$ is as given by the metric of $M_\alpha$, and the distance between distinct elements of $\bigsqcup_{\alpha<\kappa}M_\alpha$ not in the same $M_\alpha$ is~$1$.
\item For any $\alpha<\kappa$, the distance between $\alpha$ and any other element of $M$ is 1.  
\item If $Q$ is an $n$-ary predicate symbol of $L$, the interpretation $Q^{\mathcal M}$ is $\bigsqcup_{\alpha<\kappa}Q^{\mathcal M_\alpha}$ in $\bigsqcup_{i<\kappa}M_\alpha^n$ and $0$ in $M^n\setminus \bigsqcup_{\alpha<\kappa}M_\alpha^n$.
\item If $f$ is an $n$-ary operation symbol of $L$, and $\bar a \in M^n$, then $f^{\mathcal M}(\bar a)$ is  $f^{\mathcal M_\alpha}(\bar a)$ if $a\in M_\alpha^n$ and $0$ (the least element of $\kappa$) if $\bar a\in M^n\setminus \bigsqcup_{\alpha<\kappa}M_\alpha^n$.
\item $P^{\cM}$ is the characteristic function of $\kappa$.
\item $\lhd^{\mathcal M}$ is the characteristic function of $\{\,(\alpha,\beta): \alpha<\beta<\kappa\,\}$.
\item $R^{\mathcal M}$ is the characteristic function of $\bigcup_{\alpha<\kappa}(M_\alpha\times\{\alpha\})$.
\end{itemize}

Note that $\mathcal M\upharpoonright R(\cdot,\alpha)\simeq \mathcal M_\alpha$ for each $\alpha<\kappa$;
thus, $\mathcal M\upharpoonright R(\cdot,\alpha)\models\varphi$ for $\varphi\in T_{\alpha}$.
By the hypothesis of regularity of ${\sL}$ (Definitions~\ref{def:relativns} and~\ref{D:regular-logic}), 
\[
\tag{*}
\mathcal M \models   \forall y \, \Bigl( \big(P(y) \land  \overline{\alpha}\!\lhd y\big) \rightarrow \varphi^{R(\cdot,y)} \Bigr),
\]
since the collection~$\{T_{\alpha}\}$ is an ascending chain.
This shows that $\cM$ is a model of~$T$, so $T$ is satisfiable.
Moreover, if $\cM'$ is any model of~$T$, then $\{\overline{\alpha}^{\cM'}\}_{\alpha<\kappa}$ is cofinal in $(P',<') = (P^{\cM'},\lhd^{\cM'})$: 
Otherwise, there would be $\mu\in P'$ with $\overline{\alpha}^{\cM'}\lhd'\! \mu$ for all $\alpha<\kappa$, hence
\[
\mathcal M' \upharpoonright R(\cdot,\mu)\models \textstyle\bigcup\mathcal{T},
\]
by~(*), contradicting the unsatisfiability of~$\bigcup \mathcal{T}$.
\end{proof}

\begin{theorem}
\label{thm:weak-k-k-cpct-cof}

Let $\sL$ be a logic for metric structures and let $\kappa$ be a regular cardinal.
Then (1) below implies (2).
If $\sL$ is regular, (2) implies (1).
\begin{enumerate}
\item ${\sL}$ is $(\kappa,\kappa)$-compact.
\item 
As (2) of Theorem~\ref{thm:k-k-cpct-cof}, but stated for theories of~$\sL$ of cardinality~$\kappa$.
\end{enumerate}
\end{theorem}
\begin{proof}
  $(1)\Rightarrow(2)$:\quad
  The corresponding part of the proof of Theorem~\ref{thm:k-k-cpct-cof} applies verbatim.

  $(2)\Rightarrow(1)$:\quad
  The proof of Theorem~\ref{thm:k-k-cpct-cof} is adapted as follows.
  The failure of $(\kappa,\kappa)$-compactness is witnessed by an non-satisfiable family $\mathcal{T} = \{\varphi_{\alpha}\}_{\alpha<\kappa}$ of sentences all whose subfamilies of cardinality less than~$\kappa$ are satisfiable.
The theory $T = T_{\kappa} \cup \{\Discrete(R)\} \cup \bigl\{\forall y\bigl(P(y)\wedge \overline{\alpha}\lhd y\rightarrow \varphi_{\alpha}^{R(\cdot,y)} : \alpha<\kappa\bigr)\bigr\}$ of cardinality~$\kappa$ has the desired properties, by the same earlier argument.
\end{proof}

\section{Metastability and uniform metastability}
\label{sec:metastability}

This section is concerned with connections between compactness of a logic~${\sL}$ for metric structures and the notion of \emph{metastable convergence} of nets in metric spaces, i.e., in suitable structures of~${\sL}$.  
Metastability is a reformulation of the Cauchy property for nets, i.e., a net in a metric space is metastable if and only if it is Cauchy.
However, for a collection of nets, being uniformly metastable is weaker than being uniformly Cauchy.
The main results of this section are Theorem~\ref{thm:cpct-metastab} and Theorem~\ref{thm:cpct-metastab-count}, which may be roughly stated as follows:
A logic $\sL$ for metric structures is compact if and only if every theory of convergence in $\sL$ is a theory of uniformly metastable convergence.

\subsection{Metastability: Basic definitions and examples}

The following paragraphs describe the metastable viewpoint of convergence first introduced by Tao~\cite{Tao:2008}. 
The reader is referred to our earlier paper for details~\cite{Duenez-Iovino:2017}.

\begin{definition}[Samplings of a directed set]\label{def:sampling}
  A \emph{directed set} is a nonempty set~$\DD$ with a non-strict partial order~$\preceq$ such that every two elements of~$\DD$ have an upper bound but $\DD$ has no largest element.
  For $i\in\DD$, let $\DD_{\succeq i} = \{j\in\DD : i\preceq j\}$.
  A \emph{sampling} of~$\DD$ is a collection $\eta = (\eta_i : i\in\DD)$ of nonempty finite subsets of~$\DD$ such that $\eta_i\subseteq \DD_{\succeq i}$ for every $i\in\DD$.
\end{definition}

\begin{definition}[Metastable net]\label{def:metastable}
Let $(\DD,\preceq)$ be a directed set and $(Y,\dd)$ a metric space.
A \emph{$\DD$-net in~$Y$} is a collection $\ab = (a_i : i\in\DD)$ of elements of~$Y$.
A $\DD$-net $\ab$ in~$Y$ is said to be \emph{metastable} if for every $\epsilon>0$ and every sampling~$\eta$ there exists $i\in\DD$ such that $\dd(a_j,a_k) \le \epsilon$ for all $j,k\in\eta_i$.
The net~$\ab$ is said to be \emph{pointed metastable} if there is $b\in Y$ such that, for every $\epsilon>0$ and every sampling~$\eta$, there exists $i\in\DD$ such that $\dd(a_j,b) \le \epsilon$ for all $j\in\eta_i$.
In this case, $\ab$ is said to be \emph{metastable near~$b$.}
  In either case above, $i$ is said to \emph{witness} the (pointed) $[\epsilon,\eta]$-metastability (near~$b$) of~$\ab$.
\end{definition}

\begin{proposition}\label{prop:Cauchy-metastab}\hfill
  \begin{enumerate}
  \item A net is Cauchy if and only if it is metastable.
  \item A net is pointed metastable (near~$b$) if and only if it is convergent (to~$b$);
    in particular, a net is metastable near no more than one point, necessarily its limit.
\end{enumerate}
\end{proposition}
\begin{proof}
  \begin{enumerate}
  \item \cite[Proposition~1.5]{Duenez-Iovino:2017}.
  \item Routine adaptation of the proof of~(1).\qedhere
\end{enumerate}
\end{proof}

\begin{definition}\label{def:rate-metastab}
A \emph{rate of metastability~$\Eb$ for samplings of the directed set~$\DD$} is a collection $\Eb = (E_{\epsilon,\eta})$ of nonempty finite subsets of~$\DD$, indexed by positive reals $\epsilon>0$ and samplings~$\eta$ of~$\DD$.

A net $\ab$ in $(X,\dd)$ \emph{is $\Eb$-metastable} (or \emph{admits the rate~$\Eb$ of metastable convergence [near~$b$]}) if, for every $\epsilon>0$ and every sampling~$\eta$, the $[\epsilon,\eta]$-metastability [near~$b$] of~$\ab$ has a witness $i\in E_{\epsilon,\eta}$.

Given a collection $A$ of $\DD$-nets in a metric space~$(Y,\dd)$, we say that
\begin{itemize}
\item $A$ \emph{is $\Eb$-uniformly metastable} (or $A$ \emph{admits the uniform rate~$\Eb$ of metastable convergence}) if every $\ab\in A$ is $\Eb$-metastable.
\item $A$ \emph{is $\Eb$-uniformly metastable near~$\bb$} (or $A$ \emph{admits the uniform rate~$\Eb$ of metastable convergence near~$\bb$}) if $\bb$ is a collection $(b_a : \ab\in A)$ of points in~$Y$ such that every $\ab\in A$ is $\Eb$-metastable near~$b_a$.
\end{itemize}
We say that $A$~\emph{is pointed $\Eb$-uniformly metastable} if $A$ is $\Eb$-uniformly metastable near some~$\bb$.
We also say that $A$ \emph{is Cauchy} (resp., \emph{is convergent}) if every $\ab\in A$ is Cauchy (resp., is convergent)
\end{definition}

\begin{proposition}\label{prop:unif-metastability-vs-pointed}
  If $A$ is pointed $\Eb$-uniformly metastable, then $A$ is $\widetilde{\Eb}$-uniformly metastable, where $\widetilde{E}_{\epsilon,\eta} = E_{\epsilon/2,\eta}$.
\end{proposition}
The converse of Proposition~\ref{prop:unif-metastability-vs-pointed} fails, as shown by family~$B_0$ in Remarks~\ref{rem:metastability} below.
\begin{proof}
  For $\ab\in A$, if $i\in E_{\epsilon,\eta}$ witnesses the pointed $[\epsilon/2,\eta]$-metastability of~$\ab$ (necessarily near its limit~$b\in X$, by Proposition~\ref{prop:Cauchy-metastab}~(2)), the same $i$ witnesses the $[\epsilon,\eta]$-metastability of~$\ab$, since for $j,k\in\eta_i$ we have $\dd(a_j,a_k) \le \dd(a_j,b)+\dd(a_k,b) \le \epsilon/2+\epsilon/2 = \epsilon$.
\end{proof}

\begin{remarks}\hfill
\label{rem:metastability}
  \begin{enumerate}
  \item Evidently, if a net~$\ab$ admits some rate~$\Eb$ of metastability, then $\ab$ is metastable, hence Cauchy by Proposition~\ref{prop:Cauchy-metastab};
  in particular, every uniformly metastable family~$A$ is a Cauchy family that is not uniformly Cauchy.
Similarly, every pointed uniformly metastable family is a convergent family.
\item Conversely, every metastable (i.e., Cauchy) net $\ab$ admits some rate~$\Eb$ of metastability.
    Since such net is Cauchy by Proposition~\ref{prop:Cauchy-metastab}, given $\epsilon>0$ there is $i = i_{\epsilon}\in\DD$ such that $\dd(a_j,a_k)\le\epsilon$ for all $j,k\in\DD_{\succeq i}$, so it suffices to chose $E_{\epsilon,\eta} = \{i_{\epsilon}\}$. 
    Note that $E_{\epsilon,\eta}$ so chosen is independent of the sampling~$\eta$.
    Similarly, every convergent net admits some rate of metastability near its limit (and only near its limit, by Proposition~\ref{prop:Cauchy-metastab}).
\item On the other hand, if a family $A$ of nets admits a uniform rate of metastability $\Eb = (E_{\epsilon,\eta})$ such that $E_{\epsilon,\eta} = E_{\epsilon}$ is independent of the sampling~$\eta$, then for every $\epsilon>0$ there is an upper bound $i_{\epsilon}$ for the finite set~$E_{\epsilon}$.
  By considering samplings $\eta$ with $\eta_l = \{j,k\}$ for all $l\in E_{\epsilon}$, we see that $\dd(a_j,a_k)\le\epsilon$ for all $j,k\ge i_{\epsilon}$ and all $\ab\in A$, hence $A$ is uniformly Cauchy in the classical sense.
  Similarly, if $A$ is pointed $\Eb$-uniformly metastable with rates $E_{\epsilon}$ independent of the sampling, then $A$ is a convergent family whose limits are approached uniformly in the classical sense.
\item\label{item:increasing-nets} The family $B$ of non-increasing $\DD$-nets in the discrete space~$\{0,1\}$ (i.e., $\ab$ in~$\{0,1\}$ satisfying $a_i\ge a_j$ if $i\preceq j$) admits the uniform metastability rate~$\Eb = (E_{\epsilon,\eta})$ given by $E_{\epsilon,\eta} = E_{\eta} = \{k,l\}$ (independent of~$\epsilon$), where $k$ is a completely arbitrary element of~$\DD$ (e.g., the smallest element thereof, if one exists), and $l\in\DD$ is any upper bound on~$\eta_k$;
namely, for $0<\epsilon<1$, we show that the $[\epsilon,\eta]$-metastability of any $\ab\in A$ is witnessed by either $k$ or~$l$. 
Indeed, $\ab$ is either constant or not constant on~$\eta_k$:
In the former case, the $[\epsilon,\eta]$-metastability of~$\ab$ is witnessed by~$k$; 
in the latter, $a_j=0$ for some $j\in\eta_k$ so, by monotonicity and the choice of~$l$, it follows that $\ab$ is identically zero on~$\eta_l$;
thus, the metastability is witnessed by~$l$.
The family~$B$ is uniformly metastable (and hence Cauchy) but not uniformly Cauchy. 
By removing the constant net~$1_{\bullet} = (1:i\in\DD)$ from~$B$, one obtains a subfamily~$B_0$, still $\Eb$-uniformly metastable, consisting of nets that are eventually zero (hence all convergent to zero).
However, $B_0$ is \emph{not} uniformly metastable near zero:
Given any fixed finite subset~$S$ of~$\DD$, the family $B_0$ contains a net taking the constant value $1$ on~$S$.
(Of course, the full family~$B$ is not pointed uniformly metastable either.)

\item By contrast, the Cauchy family $C$ of all $\{0,1\}$-valued eventually-zero $\DD$-nets (i.e., $\DD$-nets $\ab$ for which there exists $i\in\DD$ such that $a_j = 0$ for all $j\succeq i$) is not uniformly metastable.
  (\emph{A fortiori,} $C$ is not uniformly metastable near~$0$.)  
  Indeed, given any nonempty finite subset~$S$ of~$\DD$, let $k$ be an upper bound for~$S$ in~$\DD$, and let $l\succeq k$ ($l\ne k$).
  Fix any sampling~$\eta$ of~$\DD$ such that $\eta_i = \{k,l\}$ for all $i\in S$.
  Let $\ab$ be an arbitrary $\{0,1\}$-valued eventually-zero $\DD$-net with $a_k=1$ and $a_i=0$ for $i\succeq l$.
  By construction of $k,l,\eta,\ab$, we have $\ab\in C$, but the $[\epsilon,\eta]$-metastability of~$\ab$ has no witness $i\in S$ if $0<\epsilon<1$.
  Since this holds for arbitrary nonempty $S\subseteq\DD$, we see that no rate $E_{\epsilon,\eta}$ applies to~$C$ uniformly, so $C$ is not uniformly metastable.

\item For any infinite cardinal~$\kappa$ (regarded as the ordered set of its ordinal predecessors) and ordinal $\alpha<\kappa$, consider the $\kappa$-sequence $\ab^{(\alpha)} = (a_i^{(\alpha)} : i<\kappa)$ defined by
  \begin{equation*}
    a^{(\alpha)}_i =
    \begin{cases}
      0,
      &\text{if $i\le\alpha$, and $i$ is even (or a limit ordinal),}\\
      1,
      &\text{if $i>\alpha$, or $i$ is an odd (successor) ordinal.}
    \end{cases}
  \end{equation*}
  The family $D = \{\ab^{(\alpha)} : \alpha<\kappa\}$ consists of sequences all converging to~$1$.
  However, $D$~is not pointed uniformly metastable.  
  Consider the sampling $\eta = (\eta_{\alpha} : \alpha<\kappa)$ of~$\kappa$ given by $\eta_{\alpha} = \{\alpha,\alpha+1\}$.
  For $0<\epsilon<1$, the $[\epsilon,\eta]$-metastability of~$\ab^{(\alpha)}$ near~$1$ has no witness $i<\alpha$;
  in particular, no rate of pointed $[\epsilon,\eta]$-metastability applies uniformly to the family~$D$, which is thus not pointed uniformly metastable.
  (Interchanging the roles of~$1$ and~$0$, the family~$D$ is a subfamily of $C$ in~(5) above, for the directed set $(\kappa,<)$.)
\end{enumerate}
\end{remarks}

\begin{proposition}\label{prop:UMP-pointed-vs-0}
  Fix a directed set~$(\DD,\preceq)$ and a metric space~$(Y,\dd)$.
  Let $A$ be any collection of convergent $\DD$-nets in~$Y$.
  For any $\ab\in A$, let $b_a\in Y$ denote its limit.
  Let $\widetilde{A} = \{\dd(\ab,b_a): \ab\in A\}$ be the collection of real~$\DD$-nets $\dd(\ab,b_a) = (\dd(a_i,b_a) : i\in\DD)$.
  Every net $\dd(\ab,b_a)\in \widetilde{A}$ converges to zero.
If $\widetilde{A}$ is $\Eb$-uniformly metastable near zero, then $A$ is pointed $\Eb$-uniformly metastable (with the same rate~$\Eb$, near~$\bb$).
\end{proposition}
\begin{proof}
  Immediate from the definitions.
\end{proof}

\begin{definition}\label{def:product-directed-sets}
  The \emph{product} of two directed sets~$(\DD,\preceq)$, $(\EE,\le)$ is $(\DD\times\EE,\sqsubseteq)$, where
\begin{equation*}
  (i,j)\sqsubseteq(k,l)\qquad
  \text{iff}\qquad
  \text{$i\preceq k$ and $j\le l$,}
\end{equation*}
for all $i,k\in\DD$, $j,l\in\EE$.
One verifies immediately that $(\DD\times\EE,\sqsubseteq)$ is a directed set.

Given a $\DD$-net $\ab = (a_i:i\in\DD)$ and an $\EE$-net $\bb = (b_j : j\in\EE)$, both in the same metric space~$(Y,\dd)$, their \emph{mutual distance} is the real $(\DD\times\EE)$-net
\begin{equation*}
\dd(\ab,\bb) := (\dd(a_i,b_j) : (i,j)\in\DD\times\EE).
\end{equation*}
The ($\DD\times\DD$)-net $\dd(\ab,\ab)$ will be called the \emph{self-distance of~$\ab$.}
\end{definition}

With the definition above, a net is Cauchy iff its self-distance converges to~$0$.

\begin{definition}
  An \emph{explicit majorization} for a directed set $(\DD,\preceq)$ is a mapping $(i,j)\mapsto i\ovee j$ from $\DD\times\DD$ to $\DD$ such that $i\ovee j\succeq i$ and $i\ovee j\succeq j$ for all $i,j\in\DD$.
  
  Given a sampling~$\eta$ and an explicit majorization~$\ovee$ for~$\DD$, the collection $\check{\eta} := (\eta_{i\ovee j}\times\eta_{i\ovee j} : (i,j)\in\DD\times\DD)$ is a sampling of~$\DD\times\DD$ naturally induced by~$\eta$ via~$\ovee$.

  If $S$ is any subset of~$\DD\times\DD$, let
  \begin{equation*}
    S^{\ovee} := \{i\ovee j : (i,j)\in S\}.
  \end{equation*}
  Given any rate of uniform metastability $\Eb = (E_{\epsilon,\zeta})_{\epsilon,\zeta}$ for ($\DD\times\DD$)-nets, 
the collection $\Eb^{\ovee} := \bigl((E_{\epsilon, {\check{\eta}}})^{\ovee}\bigr)_{\epsilon,\eta}$ is a naturally induced rate of uniform metastability for~$\DD$-nets.
\end{definition}

\begin{proposition}\label{prop:pointed-uni-meta-implies-uni-meta}
  Let $(\DD,\preceq)$ be a directed set with explicit majorization~$\ovee$.

  For every nonempty finite subset $S\subseteq\DD\times\DD$, all $\epsilon>0$, all samplings $\eta$ of~$\DD$ and all $\DD$-nets~$\cb$ (in some metric space), if $S$ is a rate of pointed $[\epsilon,\check{\eta}]$-metastability for the real $(\DD\times\DD)$-net $\dd(\cb,\cb)$ near~$0$, then $S^{\ovee}$ is a rate of $[\epsilon,\eta]$-metastability for~$\cb$.

In particular, if $\Eb$ is a rate of pointed metastability for~$\dd(\cb,\cb)$ near~$0$, then $\Eb^{\ovee}$ is a rate of metastability for~$\cb$.
\end{proposition}
\begin{proof}
  By the definitions of $\check{\eta}$ and~$S^{\ovee}$, if~$\dd(\cb,\cb)$ admits $S$ as a rate of $[\epsilon,\check{\eta}]$-metastability near~$0$, then there exists~$(k,l)\in S$ such that, for all~$i,j\in\eta_{k\ovee l}$, we have $\dd(c_i,c_j) \le \epsilon$.
The element $k\ovee l \in S^{\ovee}$ thus witnesses the $[\epsilon,\eta]$-metastability of~$\cb$.
\end{proof}

\subsection{Uniform metastability from a topological viewpoint}
\label{SS:UMP topological}

The following proposition is a purely topological analogue of the Uniform Metastability Principle~\cite[Proposition~2.4]{Duenez-Iovino:2017}, and of the main result of the Avigad-Iovino paper on ultraproducts and metastability~\cite[Theorem~2.1]{Avigad-Iovino:2013}.

\begin{proposition}[Topological Uniform Metastability Principle]
\label{P: Ptwise Cauchy and unif metastability}
Let $X$ be a topological space, and $\kappa$ an infinite cardinal.
  If $X$ is $[\kappa,\omega]$-compact and  $(\DD,\preceq)$ is a directed set with $\card(\DD) \le \kappa$, then the following conditions are equivalent for any $\DD$-net $\fb = (f_i : i\in\DD)$ of continuous functions from $X$ into a metric space $(Y,\dd)$:
\begin{enumerate}
\item The nets $\fb(x) = (f_i(x):i\in\DD)$ for $x\in X$ are all Cauchy.
\item The family $A = \{ \fb(x) : x \in X \}$ is uniformly metastable.
\end{enumerate}
 The following properties are also equivalent:
\begin{enumerate}[resume]
\item The nets $\fb(x)$ for $x\in X$ are all convergent.
\item The family $A$ is pointed uniformly metastable.
\end{enumerate}
In particular, if $(Y,\dd)$ is complete, properties~(1)--(4) above are all equivalent.
\end{proposition}

Without a suitable hypothesis such as that of $[\kappa,\omega]$-compactness, the preceding equivalence fails, since an arbitrary Cauchy family (which can always be regarded as a family of continuous functions on a discrete space~$X$) need not be uniformly metastable (e.g., families $C$, $D$ in Remarks~\ref{rem:metastability} above).

\begin{proof}
The implications $(1)\Rightarrow(2)$ and $(3)\Rightarrow(4)$ hold without any hypotheses on~$X$, by Propositions~\ref{prop:Cauchy-metastab} and~\ref{prop:unif-metastability-vs-pointed}, and Remark~\ref{rem:metastability}~(1).

$(4)\Rightarrow(3)$:\quad
Assume that $X$ is $[\kappa,\omega]$-compact and~$A$ is a family of convergent nets, and let $g:X\to Y$ be the pointwise limit of~$\fb$.
For the sake of contradiction, assume that $A$ is not pointed uniformly metastable.
By Proposition~\ref{prop:Cauchy-metastab}, every net~$\fb(x)$ may only be pointed metastable near its limit~$g(x)$, so $A$ must not be uniformly $[\epsilon,\eta]$-metastable near~$g$ for some $\epsilon>0$ and sampling $\eta$ of~$\DD$.
This means that no finite subset of~$\DD$ is a uniform rate of pointed $[\epsilon,\eta]$-metastability for~$A$;
thus, given a nonempty finite subset $S$ of~$\DD$, there exists $a\in X$ such that
\[
  \overline{\dd}_k(a)
  := \max\bigl\{ \dd(f_i(a),g(a)) : i\in\eta_k \bigr\}>\epsilon
\qquad{\text{for $k\in S$}};
\]
hence, $a\in\bigcap_{k\in S} C_k$, where $C_k= \{ x\in X : \overline{\dd}_k(x)\ge\epsilon \}$ is a closed subset of~$X$ by continuity of the functions~$f_i$.
Thus, the family $(C_k : k\in\DD)$ has the finite-intersection property.
Since $\card(\DD)\le\kappa$, by $[\kappa,\omega]$-compactness of~$X$ there exists an element $c\in\bigcap_{k\in\DD} C_k$. 
By construction of~$c$, the net $\fb(c) = (f_i(c) : i\in\DD)$ is not $[\delta,\eta]$-metastable, contradicting the assumption that $\fb(c)$ converges to~$g(c)$.

$(2)\Rightarrow(1)$:\quad
If necessary, identify $Y$ with a subset of its metric completion~$\overline{Y}$.
Assume the nets $\fb(x)$ are Cauchy for all $x\in X$.
Since any Cauchy net in~$Y\subseteq\overline{Y}$ converges in~$\overline{Y}$, the pointwise limit $g$ of $\fb$ exists as a function~$X\to\overline{Y}$.
By the implication $(4)\Rightarrow(3)$, the family~$A$ is pointed uniformly metastable, albeit the pointwise limit $g$ takes values in~$\overline{Y}$ and not necessarily in~$Y$.
Since the inclusion $Y\subseteq \overline{Y}$ is isometric,
the family~$A$ is uniformly metastable by Proposition~\ref{prop:unif-metastability-vs-pointed}.
\end{proof}

For nets of bounded real functions, we may strengthen Proposition~\ref{P: Ptwise Cauchy and unif metastability} to an equivalence under additional hypotheses on the domain~$X$:%
\footnote{We do not assume that regular spaces or paracompact spaces are Hausdorff; in particular, the class of paracompact spaces includes all pseudometric spaces.}

\begin{proposition}
\label{P: Ptwise Cauchy and unif metastability equivalence}
If $X$ is a regular paracompact topological space, then the following conditions are equivalent:
\begin{enumerate}
\item
$X$ is countably compact.
\item 
If $(\DD,\preceq)$ is a directed set, and $\fb = (f_i : i\in\DD)$ is a $\DD$-net of continuous functions $f_i$ from $X$ into am complete metric space, then the family $A = \{ \fb(x) : x \in X \}$ is pointed uniformly metastable if and only if each of the nets $\fb(x)$ ($x\in X$) converges.
\item 
If $\fb = (f_n : n < \omega)$ is an $\omega$-sequence of continuous functions from $X$ into $[0,1]$, then the family $A = \{ \fb(x) : x \in X \}$ is pointed uniformly metastable if and only if each of the sequences $\fb(x)$ ($x\in X$) converges.
\end{enumerate}
\end{proposition}

\begin{proof}
  $(1)\Rightarrow(2)$:\quad
  Paracompact spaces that are countably compact are compact~\cite{Dieudonne:1944}.
  Hence, assertion~(2) follows from the the equivalence of $(3)$ and~$(4)$ in Proposition~\ref{P: Ptwise Cauchy and unif metastability}.
  
  $(2)\Rightarrow(3)$:\quad
  Immediate.

  $(3)\Rightarrow(1)$:\quad
By contraposition, assume that $X$ is paracompact and regular, but not countably compact.
Then, there is a strictly increasing sequence $U_0\subsetneq U_1 \subsetneq \cdots \subsetneq U_n\subsetneq\cdots$ of nonempty open sets such that $X = \bigcup_nU_n$.
Every paracompact space is a shrinking space, so one may further assume that the collection $(\overline{U_n} : n<\omega)$ is locally finite, and $\overline{U_n}\subsetneq \overline{U_{n+1}}$~\cite[Lemma 41.6]{Munkres2000}.
By regularity, we may choose nonempty open sets~$W_0,W_1,\dots,W_n,\dots$ satisfying:
\begin{align*}
  \overline{W_0} &\subseteq U_0, \\
  \overline{W_{n+1}} &\subseteq U_{n+1}\setminus \overline{U_n},
\end{align*}
and a point $x_n\in W_n$ for each~$n$.
Clearly, $(\overline{W_n} : n<\omega)$ is a locally finite family of pairwise disjoint sets.
Every regular paracompact space is completely regular (this is an elementary exercise), so there exist continuous functions $g_n : X \to [0,1]$ such that $g_{\alpha}(x_{\alpha}) = 1$ and $g_{\alpha}(X\setminus W_{\alpha})=0$. 
For~$n<\omega$, let $h_n = \sum_{i<n}g_i$.
We may also define a function $h : X\to[0,\infty]$ by 
\[\tag{\dag}
  h(x) := \sum_{n<\omega} g_n(x)\qquad
  \text{for each $x\in X$.}
\]
Indeed, by construction of the functions~$g_n$ and the sets~$W_n$, for fixed~$x\in X$ there is at most one non-zero term in the sum on the right-hand side of equation~$(\dag)$, so $h$ and each of the functions~$h_n$ take values in~$[0,1]$.
Moreover, each $x\in X$ has an open neighborhood $U_n$ intersecting only finitely many of the sets $\overline{W_m}$, hence the supports of only finitely many of the~$g_n$'s.
Thus, $h$ is continuous on~$X$.

For arbitrary $n<\omega$, let
\begin{equation*}
  f_n =
  \begin{cases}
    h_n,
    & \text{if $n$ is odd,}\\
    h,
    & \text{if $n$ is even.}
  \end{cases}
\end{equation*}
Clearly, $\lim_nf_n(x) = h(x)$ for all~$x\in X$, so $\fb = (f_n : n<\omega)$ is pointwise convergent.
However, for fixed~$n < \omega$, the sequence $\fb(x_n) := (f_i(x_n) : i < \omega)$ is precisely the sequence~$\ab^{(n)}$ in the last of Remarks~\ref{rem:metastability} (for $\kappa = \omega$).
Thus, $(\fb(x_n) : n < \omega)$ is not pointed uniformly metastable;
\emph{a fortiori,} neither is the larger family $\{\fb(x) : x\in X\}$.
\end{proof}

\subsection{The Main Theorem: Uniform metastability and logical compactness}

In this subsection, we connect the Uniform Metastability Principle with the notion of $[\kappa,\lambda]$-compactness for logics introduced in Section~\ref{S:prelim}-\ref{S:compactness}.

\begin{definition}
\label{D:convergence modulo:T}
Let $(\DD,\preceq)$ be a directed set, let ${\sL}$ be a logic for metric structures, and let $T$ be a uniform $L$-theory of ${\sL}$, where $L$ is a vocabulary.

Given a $\DD$-net $\Pb = (\varphi_i : i\in\DD)$ of $L$-sentences, we say that $\Pb$ is
\begin{itemize}
\item \emph{Cauchy (convergent) modulo~$T$} if the $\DD$-net $\Pb^{\cM} = (\varphi_i^{\cM} : i\in\DD)$ in~$[0,1]$ is Cauchy (convergent) in every model $\cM$ of~$T$.
  If $\Pb^{\cM}$ is convergent, let $\lim\Pb^{\cM}$ denote its limit; we shall say that $\Pb^{\cM}$ converges to it.
\item \emph{uniformly metastable modulo~$T$} if there exists a metastability rate that applies uniformly to all nets $\Pb^{\cM}$ obtained from models $\cM$ of~$T$.  
\item  \emph{pointed uniformly metastable modulo~$T$} if all nets~$\Pb^{\cM}$ are convergent, and a fixed metastability rate applies to~$\Pb^{\cM}$ near~$\lim\Pb^{\cM}$ as $\cM$ ranges over models of~$T$.
\end{itemize}
Given a $\DD$-indexed family $\tb = (t_i : i\in\DD)$ of $L$-terms, we say that $\tb$ is
\begin{itemize}
\item \emph{Cauchy (convergent) modulo~$T$} if the net $\tb^{\cM} = (t_i^{\cM}:i\in\DD)$ is Cauchy (convergent) in every model~$\cM$ of~$T$;
  the limit of such a Cauchy net (taken in the metric completion~$\overline{\cM}$ of~$\cM$ if necessary) is denoted~$\lim\tb^{\cM}$.
\item \emph{uniformly metastable modulo~$T$} if  the collection of all nets $\tb^{\cM}$, as $\cM$ varies over all models of~$T$, is uniformly metastable
\item \emph{pointed uniformly metastable modulo~$T$} if all nets $\tb^{\cM}$ are convergent, and a fixed metastability rate applies to $\tb^{\cM}$ near $\lim\tb^{\cM}$ as $\cM$ varies over all models of~$T$.
\end{itemize}
\end{definition}

The implication $(1)\Rightarrow(3)$ of the following theorem is a generalization of the Uniform Metastability Principle for continuous first-order logic~\cite{Duenez-Iovino:2017}.

\begin{theorem}\label{thm:cpct-metastab}
  Let $\sL$ be a regular logic for metric spaces.
  The following properties are equivalent for any infinite cardinal~$\kappa$:
\begin{enumerate}
\item $\sL$ is $[\kappa,\omega]$-compact.
\item The following pointed Uniform Metastability Principle holds for the logic~$\sL$ and all directed sets~$(\DD,{\preceq})$ of cardinality at most~$\kappa$:
  If $T$ is a uniform $L$-theory of $\sL$ and $\Pb$ is a $\DD$-net consisting of no more than $\kappa$-many $L$-sentences, then $\Pb$ is Cauchy modulo~$T$ if and only if $\Pb$ is pointed uniformly metastable modulo~$T$.
\item The following Uniform Metastability Principle holds for the logic $\sL$ and all directed sets~$(\DD,\preceq)$ of cardinality at most~$\kappa$:
  If $\tb = (t_i : i\in\DD)$ is a collection of $L$-terms, and $T$ is any uniform $L$-theory of~$\sL$, then $\tb$ is Cauchy modulo~$T$ if and only if $\tb$ is uniformly metastable modulo~$T$.
\end{enumerate}
\end{theorem}
\begin{proof}
$(1)\Rightarrow(2)$:\quad
Assume that $\sL$ is $[\kappa,\omega]$-compact.
Fix a directed set $(\DD,\preceq)$ of cardinality at most~$\kappa$, let $\Pb$ be a $\DD$-net of $\sL$-sentences, and let $T$ be a uniform $L$-theory of $\sL$ such that $\Pb$ is convergent modulo~$T$. 
The space $X = \Mod_{\sL}(T)$ of models of~$T$ carries the logic topology whose sub-basic closed sets are those of the form $\Mod_{\sL}(\psi) = \{\cM\in X : \cM\models\psi\}$, where $\psi$ is an $L$-sentence.
Since $\sL$ is closed under the basic connectives, each mapping $\overline{\psi} : X\to[0,1]$ given by $\cM\mapsto\psi^{\cM}$ (the truth value of~$\psi$ in~$\cM$) is continuous because, for any closed sub-basic $[a,b]\subseteq [0,1]$, we have
\[
  \{ \cM \in X : \psi^{\cM} \in \lbrack a,b] \}
  = \Mod_{\sL}(a \leq \psi \wedge \psi \leq b).
\]
Thus,
$\Pb$ is identified with the $\DD$-net $(\overline{\varphi_i} : i\in\DD)$ of continuous functions $X\to[0,1]$.
Assertion~(2) now follows from Proposition~\ref{P: Ptwise Cauchy and unif metastability}.

$(2)\Rightarrow(3)$:\quad
Assume that the Uniform Metastability Principle~(2) holds in~$\sL$ for directed sets of cardinality at most~$\kappa$.
Let $L,\tb,T$ satisfy the hypotheses of~(3) for some directed set~$(\DD,\preceq)$ of cardinality at most~$\kappa$.
Fix an explicit majorization~$\ovee$ for~$\DD$.
Since $\DD$ is infinite, $\card(\DD\times\DD) = \card(\DD) \le \kappa$.
Let $\varphi_{ij}$ be the formula $\dd(t_i,t_j)$ for $(i,j)\in\DD\times\DD$.
The hypothesis that $\tb$ is Cauchy modulo~$T$ implies that $\Pb$ converges to~$0$ modulo~$T$.
By assumption~(2) (the pointed UMP), $\Pb$ is uniformly metastable near~$0$ modulo~$T$;
thus, the $(\DD\times\DD)$-nets $\dd(\tb^{\cM},\tb^{\cM})$ all admit a uniform rate $\Eb = (E_{\epsilon,\zeta})$ of metastability near zero as $\cM$ varies over models of~$T$.
By Proposition~\ref{prop:pointed-uni-meta-implies-uni-meta}, $\tb$ admits the uniform rate of metastability~$\Eb^{\ovee}$ modulo~$T$.
This proves that $\Pb$ is $\Eb$-uniformly metastable modulo~$T$.

$(3)\Rightarrow(1)$:\quad
By contraposition, assume that $\sL$ is not $[\kappa,\omega]$-compact.
Then $\sL$ is not $[\mu,\mu]$-compact for some regular $\mu\le\kappa$, by the last of Remarks~\ref{R:countable compactness of logics}.
Let $T$ be a uniform $L$-theory extending $T_{\mu}$ in part~(2) of Theorem~\ref{thm:k-k-cpct-cof} such that $(\overline{\alpha}^{\cM}:\alpha<\mu)$ is cofinal in $(P^{\cM},\lhd^{\cM})$ for all models~$\cM$ of~$T$.
Expand $L$ to a vocabulary $L'$ having a unary function symbol $f$ and a family $\cb = (c_{\alpha} : \alpha<\mu$) of constants not used by the theory~$T$.
Let $T'$ be the union of $T$ and the axioms:
\begin{subequations}
  \begin{align}
    &\forall x\bigl(\dd(f(x),\overline{0})=0\vee \dd(f(x),\overline{1})=0\bigr),\label{eq:axiom1}\\
    &\forall x\bigl(\lnot P(x)\rightarrow \dd(f(x),\overline{0})=0\bigr),\label{eq:axiom2}\\
    &(\exists x\in P)(\forall y\in P)\bigl(x<y\rightarrow \dd(f(y),\overline{0})=0\bigr),\label{eq:axiom3}\\
    &\dd(c_{\alpha },f(\overline{\alpha }))=0,\qquad
      \text{for each $\alpha <\mu$.}\label{eq:axiom4}
  \end{align}
\end{subequations}
$T'$ is a uniform~$L'$-theory since $T$ (hence~$T'$) implies a modulus of uniform continuity for interpretations of~$P$ by hypothesis (thus, since $P$ is discrete, $\{x : P(x)=0\}$ is uniformly separated from $\{x : P(x)=1\}$), and clearly~$T'$ implies the same modulus of uniform continuity for interpretations of~$f$.
Note that if  $\cM'$ is any model of~$T'$, then the $\mu$-sequence~$\cb^{\cM'} = (c_{\alpha}^{\cM}:\alpha<\mu)$ is eventually zero, by axioms~\eqref{eq:axiom3} and~\eqref{eq:axiom4} above.

Consider an arbitrary eventually-zero $\mu$-sequence $\bb = (b_{\alpha}:\alpha<\mu)$ in~$\{\overline{0},\overline{1}\}$.
Since $(\mu,<)$ is well ordered, there is a least $\gamma<\mu$ such that $b_{\alpha}=0$ for all $\alpha\ge\gamma$.
Since $T$ uses neither~$f$ nor the constants~$c_\alpha$, every model~$\cM$ of~$T$ admits an expansion to a model $\cM' = \cM'_{\bb}$ of~$T'$ as follows.
Let $c_\alpha^{\cM'} = b_{\alpha}$ for $\alpha<\mu$, and let $f^{\cM'}$ be the following function $M \to\{\overline{0},\overline{1}\}$ that takes the value~$\overline{0}$ outside~$P^{\cM}$ and, for $x\in P^{\cM}$,
\begin{itemize}
\item $f(x) = b_{\alpha}$, if $x = \overline{\alpha}^{\cM}$ for some~$\alpha<\mu$,
\item $f(x) = \overline{0}$, if $x\in P^{\cM}$ and $x\ge\overline{\gamma}^{\cM}$,
\item $f(x) = \overline{1}$, for all other elements $x\in P^{\cM}$.
\end{itemize}

Conversely, in any model~$\cM'$ of~$T'$ the map~$f^{\cM'}$ is $\{\bar 0,\bar1\}$-valued and eventually zero in $(P^{\cM'},\lhd^{\cM'})$, by axioms~\eqref{eq:axiom1} and~\eqref{eq:axiom3} above.
Since $(\overline{\alpha}^{\cM'}:\alpha<\mu)$ is cofinal in~$P^{\cM'}$, axiom~\eqref{eq:axiom4} implies that the binary sequence~$\cb^{\cM'}$ is also eventually zero.
In sum, when $\cM'$ runs through all models of~$T'$, the sequence $\cb^{\cM'}$ runs through the family~$C$, introduced in Remarks~\ref{rem:metastability}, of all binary sequences of length~$\mu$ with limit zero.
The family~$C$ is Cauchy but not uniformly metastable, so $\cb$ is Cauchy but not uniformly metastable modulo~$T'$.
\end{proof}

Theorem~\ref{thm:cpct-metastab} characterizes $[\kappa,\omega]$-compactness.
We now observe that a slight modification of its proof yields an analogous characterization of the weaker property of $(\kappa,\omega)$-compactness (see Remark~\ref{rem:contrast-countable-compactnesses}):

\begin{theorem}\label{thm:cpct-metastab-count}
    Let $\sL$ be a regular logic for metric spaces.
  The following properties are equivalent for any infinite cardinal~$\kappa$:
\begin{enumerate}
\item $\sL$ is $(\kappa,\omega)$-compact.
\item
As in (2) of Theorem~\ref{thm:cpct-metastab}, with the additional condition that the cardinality of $T$ is at most~$\kappa$.
\item
As in (3) of Theorem~\ref{thm:cpct-metastab}, with the additional condition  that the cardinality of $T$ is at most~$\kappa$.
\end{enumerate}
\end{theorem}
\begin{proof}
  $(1)\Rightarrow(2)$:\quad
  Let $X = \Mod_{\sL}(T)$,
  and for each $L$-sentence $\psi$, let $\overline{\psi} : X\to[0,1]$ be the mapping given by $\cM\mapsto\psi^{\cM}$.
  Give $X$ the initial topology making all mappings~$\overline{\varphi_i} : X\to[0,1]$ ($i\in\DD$) continuous (which is coarser than that in the proof of Theorem~\ref{thm:cpct-metastab}).
  The collection of all the classes $\{\cM\in X : \varphi_i^{\cM} \in [a,b]\}$ for $i\in\DD$ and~$a, b$ rational with $0\le a\le b\le 1$ is a subbasis of cardinality~$\le\kappa$ for the closed sets of this topology.
  Since~ $|T|\le\kappa$, the $(\kappa,\omega)$-compactness of~$\sL$ implies that every countable collection of $L$-sentences of cardinality~$\le\kappa$ that is finitely jointly satisfiable with~$T$ is jointly satisfiable with~$T$.
  By considering collections of sentences of the form ``$a\le\varphi_i\le b$'', we see that any family of subbasic closed sets having the finite-intersection property has nonempty intersection;
  this suffices to conclude that $X$ is compact.
  Assertion~(2) now follows from Proposition~\ref{P: Ptwise Cauchy and unif metastability} as in the proof of Theorem~\ref{thm:cpct-metastab}.

  $(2)\Rightarrow(3)$:\quad
  Identical to the corresponding part of the proof of Theorem~\ref{thm:cpct-metastab}.

  $(3)\Rightarrow(1)$:\quad
  Assume that $\sL$ is regular and not $(\kappa,\omega)$-compact, and let $\mu$ be least such that $\sL$ is not $(\mu,\omega)$-compact.
  It is easy to see that $\sL$ is not $(\mu,\mu)$-compact and $\mu$ is regular.
  By Theorem~\ref{thm:weak-k-k-cpct-cof} there is a satisfiable theory~$T$ extending~$T_{\mu}$ with $|T|\le\mu$ in all whose models~$\cM$ the $\mu$-sequence $\cb^{\cM}$ is cofinal in~$(P^{\cM},\lhd^{\cM})$.
  The rest of the proof is identical to the corresponding one in Theorem~\ref{thm:cpct-metastab}.
\end{proof}

\begin{remark}
In Theorems~\ref{thm:cpct-metastab} and \ref{thm:cpct-metastab-count}, the regularity assumption on the logic is needed only  in $(3)\Rightarrow(1)$ to invoke Theorem~\ref{thm:k-k-cpct-cof}.
The other implications hold for arbitrary logics. The forward direction of both theorems, namely, $(1)\Rightarrow(2)$, it sufficed to invoke the topological version of the UMP 
(Proposition~\ref{P: Ptwise Cauchy and unif metastability}).
However, for the converse, we could not rely solely on topology because topological spaces that arise from logics need not be paracompact (see~\cite{Caicedo:1993} for examples).
\end{remark}

\section{Compactness and \RPCd-characterizability of general structures}
\label{S:characterizability}

If $L,L'$ are vocabularies with $L\subseteq L'$ and ${\cM}$ is an $L'$-structure, the $L$-reduct of ${\cM}$ will be denoted $\mathcal M\upharpoonright L$. 
We will use the restriction symbol $\upharpoonright$ with a second meaning: 
If ${\cM}$ is an $L$-structure and $A$ is a subset of the universe of ${\cM}$ that is $L$-closed, i.e., for every $n<\omega$, the set $A^n$ is closed under all the $n$-ary functions of ${\cM}$, then $\mathcal M\upharpoonright  A$ will denote the substructure of ${\cM}$ that results from restricting, for every $n<\omega$, all the $n$-ary functions and predicates of ${\cM}$ to~$A$.

Definitions \ref{D:PC} and \ref{D:RPC} below are classical.

\begin{definition}
\label{D:PC}
Let ${\sL}$ be a logic for metric structures and let $L$ be a vocabulary.  
A class $\mathscr{C}$ of $L$-structures is said to be a \emph{projective class} (or \emph{PC}) in ${\sL}$ if there exists a vocabulary $L'\supseteq L$ and an ${\sL}$-elementary class $\mathscr{C}'$ of $L'$-structures such that
\[
\mathscr{C}=\{ \mathcal M\upharpoonright L : \mathcal M\in\mathscr{C}'\}.
\]
\end{definition}

Intuitively, a class is PC in $\cl$ if and only if it is definable by an existential second-order sentence.  
Thus, the concept of projective class can be seen as a generalization of this notion of definability to arbitrary logics.

\begin{definition}
\label{D:RPC}
Let ${\sL}$ be a logic for metric structures and let $L$ be a vocabulary. 
A class $\mathscr{C}$ of $L$-structures is a \emph{relativized projective class} (or \emph{RPC}) in ${\sL}$ if there exist a vocabulary $L'\supseteq L$, an ${\sL}$-elementary class $\mathscr{C}'$ of $L'$-structures, and a monadic predicate symbol $R$ of $L'$ such that $R$ is discrete in each structure $\mathcal{M}\in\mathscr{C}'$, and
\[
\mathscr{C} = \{ (\mathcal M\upharpoonright L)\upharpoonright R^{\cM} : \text{$\mathcal M\in\mathscr{C}'$ and $R$ is $L$-closed in ${\cM}$} \}.
\]

The notions of {\PCd} and \RPCd\ are defined by replacing ``${\sL}$-elementary'' with ``${\sL}$-axiomatizable'' in the definitions of PC and RPC, respectively.
\end{definition}

\begin{proposition}
\label{prop:cpct-implies-nonrpc}
If~$\kappa$ is a regular cardinal and $\sL$ is $[\kappa,\kappa]$-compact logic for metric structures, then the structure $(\kappa,<)$ is not \RPCd in~${\sL}$.
\end{proposition}

\begin{proof}
Assume that $\sL$ is $[\kappa,\kappa]$-compact. By Theorem~\ref{thm:k-k-cpct-cof}, any \RPCd
definition of  $(\kappa,<)$ via discrete predicate symbols $P$ and $\lhd$ must have a model $\cM$ such that $(\kappa,<)$ embeds non-cofinally into $(P^{\cM}, \lhd^{\cM})$.
Hence, $(P^{\cM}, \lhd^{\cM})$ has a proper initial segment of cardinality at least $\kappa$.
This prevents $(P^{\cM}, \lhd^{\cM})\cong (\kappa,<)$.
\end{proof}

\begin{proposition}
\label{prop:nonrpc-implies-cpct}
Let ${\sL}$ be a regular logic for metric structures.
If the structure $(\omega,<)$ is non \RPCd\ in~${\sL}$, then $\sL$ is $[\omega,\omega]$-compact.
\end{proposition}

\begin{proof}

Assume that $\sL$ is not $[\omega,\omega]$-compact.
By Theorem~\ref{thm:k-k-cpct-cof},
there exists an $\sL$-theory $T$ in a vocabulary $L$ containing
a monadic predicate symbol $P$, a binary predicate symbol $\lhd$, and a family $(c_{n})_{n<\omega}$ of constant symbols
such that if $\cM=(M,P^{\cM},<,c_n^{\cM} \dots)_{n<\omega}$ is a model of $T$, then $P^{\cM}$ is discrete, $(P^{\cM},<)$ is a discrete linear order and $(c_n^{\cM})_{n<\omega}$ is a cofinal sequence in $(P^{\cM},<)$.
 Let $L'$ be an expansion of $L$ that  contains a new monadic predicate symbol $R$
and let $T'$ be the $L'$-theory that consists of ~$T$ plus the following sentences:
\begin{itemize}
\item
$\Discrete(R)$,
 \item
 $\forall \bar{x}(R(\bar{x}) \to P(\bar{x}))$,
 \item
 $R(c_n)$, for each $n<\omega$,
\item
$\forall x( (c_n\lhd x\lhd c_{n+1})\to \neg R(x))$, for each $n<\omega$.
 \end{itemize}
 If $\cN$ is a model of $T'$, then the restriction $(\cN\upharpoonright L)\upharpoonright R$ is isomorphic to $(\omega<)$.
\end{proof}

\begin{corollary}
\label{C:noncpct-rpc-omega}
The following conditions are equivalent for any regular logic ${\sL}$ for metric structures:
\begin{enumerate}
\item ${\sL}$ is $[\omega,\omega]$-compact.
\item The structure $(\omega,<)$ is not \RPCd\ in~${\sL}$.
\end{enumerate}  
\end{corollary}
\begin{proof}
Immediate from Propositions~\ref{prop:cpct-implies-nonrpc} and~\ref{prop:nonrpc-implies-cpct}.
\end{proof}

The goal of the rest of this section is to prove an analog of the preceding corollary with  $(\omega,<)$ replaced with an arbitrary infinite structure.

\begin{definition}
 If $\sL$ is a logic for metric structures, a structure $\mathcal M$ of $\sL$ with universe $M$ will be called \emph{full} if, for every $n<\omega$, every uniformly continuous function $M^n\to M$ 
 and every uniformly continuous predicate $M^n\to [0,1]$ is definable in~$\thry_{\sL}(\cM)$.
\end{definition}

The following lemma is a slight improvement of a theorem of Makowsky and Shelah~\cite[Theorem 1.1]{Makowsky-Shelah:1983}.

\begin{lemma}
\label{L:nonprincipal ultrafilter}
Let $\mathcal M$ be a discrete full structure with universe $M$, and let $\Hat{{\cM}}$ be an ${\sL}_{\omega\omega}$-extension of ${\cM}$.  
Let $P$ be a subset of~$M$ and assume that the natural extension $\Hat{P}$ of~$P$ to~$\Hat{\cM}$ is a proper superset of~$P$.
Fix $b\in\Hat{P}\setminus P$. 
Define
\[
\mathcal{U}=\{\, R\subseteq P : b\in\Hat{R}\setminus R\,\}.
\]
\begin{enumerate}
\item
$\mathcal{U}$ is a nonprincipal ultrafilter over $P$.
\item
If the cardinality of $P$ is less than the first measurable cardinal, then $\Hat{Q}\setminus Q\neq \emptyset$ for every infinite $Q\subseteq P$.
\end{enumerate}
\end{lemma}

\begin{proof}
  (1): The fact that $\mathcal{U}$ is an ultrafilter follows from the following assertions, which can be verified immediately:
  \begin{itemize}
  \item $b\in\Hat{P}$, but $b\notin\emptyset=\Hat{\emptyset}$,
  \item for all $R,S\subseteq P$:
    \begin{itemize}
    \item  $R\subseteq S$ if and only if $\Hat{R}\subseteq\Hat{S}$,
    \item $(R\cap S)\,{\Hat{}} = \Hat{R}\cap\Hat{S}$, and $(R\cup S)\,{\Hat{}}= \Hat{R} \cup \Hat{S}$.
  \end{itemize}
\end{itemize} 
To see that $\mathcal{U}$ is nonprincipal, notice that if $\mathcal{U}$ were generated by a singleton, say $R = \{a\}$ with $a\in M$, then clearly $\Hat{R} = \{a\} = R$, while by definition of $\mathcal{U}$ we would have $b\in \Hat{R}\setminus R = \emptyset$: A contradiction.

(2):
Fix an infinite $Q\subseteq P$ with $\card(Q) = \mu$ and assume $\Hat{Q} = Q$.  
We will show that this implies that $\mathcal{U}$ is a $\mu$-complete ultrafilter.  
By a well-known characterization of measurability, it will follow that the cardinality of $Q$ must be at least the first measurable cardinal.

Fix a cardinality-$\mu$ family $(R_q)_{q\in Q}$ of elements of $\mathcal{U}$, assumed indexed by~$Q$ without loss of generality.
For each $q\in Q$, let $\chi_q$ be the symbol for the characteristic function of~$R_q$;
similarly, let $\psi$ be the symbol for the characteristic function of~$Q$.

Let $\xi$ be the symbol for the characteristic function of the set $\{(x,q) : x\in R_q\} \subseteq M^2$.
Then ${\cM}$ and its elementary extension $\Hat{{\cM}}$ both satisfy the following sentences, for each $q\in Q$:
\begin{itemize}
\item
$\forall x \big(\chi_q(x)\rightarrow \xi(x,q)\big)$,
\item
$\forall x \Big(\forall y\big(\psi(y)\rightarrow \xi(x,y)\big)\rightarrow\chi_q(x)\Big)$;
\end{itemize}
hence,
\[
  \left\{ a\in \Hat{M} : a\in \bigcap_{q\in Q} \Hat{R}_q \right\}
    = \left\{ a\in \Hat{M} : \Hat{{\cM}}\models \forall y\,\xi(a, y) \right\}.
\]
Since $Q$ is definable, we have $\left(\bigcap_{q\in Q} R_q\right)\Hat{} =\bigcap_{q\in \Hat{Q}} \Hat{R}_q$.
On the other hand, since $Q=\Hat Q$ by hypothesis and $b\in \Hat{R}_q$ for each $q\in Q = \Hat{Q}$ by definition of~$\mathcal{U}$,  we have $b\in\bigcap_{q\in Q} \Hat{R}_q = \left(\bigcap_{q\in Q} R_q\right)\Hat{}$\,, i.e., $\bigcap_{q\in Q} R_q\in \mathcal{U}$.
\end{proof}

\begin{theorem}
\label{T:rpcdelta-characterizability of structures}
Let $\mathscr L$ be a regular logic for metric structures.
If $\mathscr L$ is not $[\omega,\omega]$-compact, then any complete structure~$\cM$ of cardinality strictly less than the first measurable cardinal is \RPCd-characterizable in~$\mathscr L$.
(If no measurable cardinal exists, the conclusion holds for all complete metric structures.)
\end{theorem}

\begin{proof}
  Assume $\sL$ is not $[\omega,\omega]$-compact and let $\cM$ be any complete structure of cardinality~$\mu$ less than the first measurable cardinal.
  Let $L$ be the vocabulary for~$\cM$.
  
By Proposition~\ref{prop:nonrpc-implies-cpct}, let $T$ be a satisfiable theory in a vocabulary $L'$ (assumed disjoint from $L$ without loss of generality) that contains predicate symbols $P$ (monadic) and $\vartriangleleft $ (binary), and such that all models $\cN$ of $T$ have discrete interpretations $P^{\cN}$, $\vartriangleleft ^{\cN}$ with $(P^{\cN},\vartriangleleft ^{\cN})$ isomorphic to $(\omega, <)$.
  Fix one such model $\cN$.
  
  Let ${\cK}$ be the structure (on a vocabulary $\widetilde{L}$ extending both $L$ and $L'$) obtained as the full expansion of the following structure:

\begin{itemize}
\item The universe $K=\mu \sqcup M\sqcup N$ is the disjoint union of the cardinal $\mu$ and the universes of ${\cM}$ and ${\cN}$.

\item The metric $\dd$ of $\cK$ extends $\dd^{\cN}$ and $\dd^{\cM}$ discretely (i.e., it is discrete in $\mu$ and separates by 1 the three disjoint parts $\mu, M, N$ of $K$).

\item If $R$ is an $n$-ary predicate symbol of $L$ $($respectively $ L')$, then the interpretation $R^{\cK}$ is equal to $R^{\cM}$ on $M^{n}$ (respectively $R^{\cN}$ on $N^{n})$ and is the constant 1 on the rest of $K^{n}$.
  This applies, in particular, to the predicate symbols $ P $ and $\vartriangleleft $.

\item If $f$ is an $n$-ary function symbol of~$L$ (respectively, of~$L')$ then $f^{\cK}$ is $f^{\cM}$ on $M^{n}$ (respectively $f^{\cN}$ on $N^{n})$ and it is $0\in \mu $ in the rest of $K^{n}$.
 \end{itemize}
 
 Since ${\cK}$ is full, for every $n<\omega $, $\widetilde{L}$ has symbols interpreted as each uniformly continuous function $K^n\rightarrow \lbrack 0;1]$;
 in particular, $\widetilde{L}$ has symbols $\chi _{M}$, $\chi _{N},$ $\chi _{\mu }$ interpreted as the characteristic functions of $M$, $N$, $\mu$, respectively.
 Moreover, $\widetilde{L}$ contains a symbol $F$ interpreted as a function $f:\mu \rightarrow M$ which maps $\mu$ bijectively onto $M$, and maps $K\setminus M$ to $0\in \mu$.

Let now ${\cK}^{*}$ be any complete ${\sL}$-elementary extension of ${\cK}$, with universe $K^{*}$, and call $M^{*}$, $N^{*}$, $P^{*}$, $\vartriangleleft^{*}$, $\mu^{*}$, $f^{*}$, respectively, the extensions in ${\cK}^{*}$ of the discrete predicates $M$, $N$, $P$, $\vartriangleleft$, $\mu$ and the map $f$.

{\bf Claim 1}. $\mu$ does not grow, that is, $\mu ^{*}=\mu$.
The restriction ${\cK}^{*}\upharpoonright N^{*}$ is an $\sL$-elementary extension of ${\cK}\upharpoonright N$. Hence, by the choice of ${\cN}$, $(P^{*},\vartriangleleft ^{*})$ is isomorphic to $(\omega ,<)$ and thus $P^{*} $ is identical to $P$. 
Also ${\cK}^{*}\upharpoonright N^{*}\sqcup\mu ^{*} $ is a $\sL$-elementary extension of ${\cK}\upharpoonright N\sqcup\mu$ and thus, in particular, a $\sL_{\omega \omega }$-elementary extension with respect to the discrete full expansion of $N\sqcup \mu $.
Therefore, we have $\mu ^{*}=\mu$, by Lemma~\ref{L:nonprincipal ultrafilter}.

{\bf Claim 2}. ${\cM}$ does not grow, that is, ${\cM}^{*}={\cM}$. 
Since $f(\mu )=M$, the relativized formula
\[
(\forall x\in M)(\exists i\in \mu )(\dd(f(i),x)\leq \epsilon) 
\]
holds in~$\cK$ for any $\epsilon>0$, and thus also in the $\sL_{\omega\omega}$-elementary extension $\cK^*\supseteq\cK$, i.e.,
\[
{\cK}^{*}\models (\forall x\in M^{*})(\exists i\in \mu^{*})(\dd(f^{*}(i),x)\leq \epsilon).
\]
However, $\mu ^{*}=\mu$; thus, $M=f(\mu )=f^{*}(\mu )$ is dense in $(M^{*},d^{*})$.
Being complete, $M$ is closed in $M^{*}$, hence $M^{*}=M$.

The last claim implies that $\thryL(({\cK},a)_{a\in K})$ gives an \RPCd-characterization of~${\cM}$ since the relativization by~$\chi_M$ of any model thereof must be isomorphic to~$\cM$.
\end{proof}

\begin{remark}
\label{R:measurable}
If $\kappa$ is the first measurable cardinal, then the two-valued infinitary logic $\sL_{\kappa\kappa}$ is not $[\omega,\omega]$-compact. However, since every measurable cardinal is weakly compact, $\sL_{\kappa\kappa}$ is  $[\kappa,\kappa]$-compact;
hence, the structure $(\kappa,<)$ is not RPC$_\Delta$ in $\sL_{\kappa\kappa}$.
\end{remark}

\begin{corollary}\label{thm:Makowsky-Shelah-continuous}
  Let $\sL$ be a regular logic for metric structures.
  Let $\kappa$ be the smallest regular cardinal such that $\sL$ is $[\kappa,\kappa]$-compact ($\kappa = \infty$ if no such cardinal exists).
  Either $\kappa=\omega$, or $\kappa$ is at least as large as the first measurable cardinal.
  In particular, if no measurable cardinals exist, we must have $\kappa = \omega$ or $\kappa = \infty$.
\end{corollary}
\begin{proof}
  If $\kappa$ is less than the first measurable cardinal and $\sL$ is not $[\omega,\omega]$-compact, then for any infinite regular cardinal $\lambda<\kappa$, the structure $(\lambda,<)$ is \RPCd-characterized in~$\sL$ by Theorem~\ref{T:rpcdelta-characterizability of structures}.
  Hence, $\sL$ is not $[\lambda,\lambda]$-compact, by Proposition~\ref{prop:cpct-implies-nonrpc} in contrapositive form.
\end{proof}

Corollary~\ref{thm:Makowsky-Shelah-continuous} is a generalization to the metric setting of a theorem of Makowsky and Shelah~\cite{Makowsky-Shelah:1983}.

If ${\sL}$ is a logic for metric structures and $(\DD,\preceq)$ a fixed directed set, let us say that \emph{the Uniform Metastability Principle (UMP) holds in~$\sL$ for $\DD$-nets} if the equivalence asserted in part~(3) of Theorem~\ref{thm:cpct-metastab} holds.

\begin{theorem}
\label{T:descent from nets to sequences}
  Let ${\sL}$ be a regular logic for metric structures and let $\kappa$ be a cardinal less than the first measurable.
  If the Uniform Metastability Principle holds for $\kappa$-sequences, then it holds for $\omega$-sequences.
\end{theorem}

\begin{proof}
  Assume that $\kappa$ is less than the first measurable and that the UMP for $\kappa$-sequences.
  If $\sL$ were not $[\omega,\omega]$-compact, then $\kappa$ would be \RPCd\ in $\sL$ by Theorem~\ref{T:rpcdelta-characterizability of structures}.
  Proceeding as in the proof of the implication $(3)\Rightarrow(1)$ of Theorem~\ref{thm:cpct-metastab}, we would find a theory of $\sL$ for a class of Cauchy $\kappa$-sequences that is not uniformly metastable.
  Thus, $\sL$ must be~$[\omega,\omega]$-compact, so UMP in~$\sL$ holds for sequences, by Theorem~\ref{thm:cpct-metastab}.
\end{proof}

\begin{remark}
	\label{nets-large-to-small}
  In general, the Uniform Metastability Principle does not descend from large nets to small nets.
  For example, if $X$ is an infinite set and $\kappa$ is a regular cardinal such that  $\kappa > |X|$, then any $\kappa$-sequence $\fb = (f_\alpha:\alpha<\kappa)$ of maps $f_\alpha:X\to\mathbb{R}$ such that $\fb$ converges to some $f : X\to\RR$ pointwise must converge uniformly, and thus $(F_\alpha:\alpha<\kappa)$ is uniformly metastable.
  (To see this, given $\epsilon>0$,  for each $x\in X$ choose $\beta_x<\kappa$ such that $|f_\alpha(x)-f(x)|\le\epsilon$ for every $\alpha>\beta_x$;
  by regularity, $\beta := \sup_{x\in X}\beta_x < \kappa$, and thus $|f_\alpha(x)-f(x)|\le\epsilon$ for every $\alpha>\beta$ and every $x\in X$.)
  On the other hand, it is easy to exhibit a pointwise converging sequence $\fb = (f_n)_{n<\omega}$ that is not uniformly metastable, e.g., by adapting the construction of family~$D$ in Remarks~\ref{rem:metastability}:
  If $(x_i)_{i<\omega}$ is a countable collection of distinct points of~$X$, let $f_n(x) = 1$ if both (\emph{i})~$n$ is odd, and (\emph{ii})~$x=x_i$ for some $i>n$, and let $f_n(x) = 0$ otherwise.
  Then $f_n\to 0$ pointwise, but $\{\fb(x):x\in X\}$ is not uniformly metastable.
\end{remark}


\def\cprime{$'$}
\providecommand{\bysame}{\leavevmode\hbox to3em{\hrulefill}\thinspace}
\providecommand{\MR}{\relax\ifhmode\unskip\space\fi MR }
\providecommand{\MRhref}[2]{%
  \href{http://www.ams.org/mathscinet-getitem?mr=#1}{#2}
}
\providecommand{\href}[2]{#2}

\end{document}